\documentclass[sn-mathphys-num]{sn-jnl}
\usepackage{float}
\usepackage[T1]{fontenc}
\usepackage{epsfig}
\usepackage{graphicx}
\usepackage{afterpage}
\usepackage{tikz}
\usepackage{orcidlink}
\usepackage{xcolor}
\definecolor{OrangeRed}{RGB}{255, 69, 0}
\usepackage{color}

\usepackage{amssymb}

\usepackage{amsmath}
\usepackage{lipsum}
\usepackage{comment}

\usepackage[utf8]{inputenc}

\usepackage{amsfonts}
\usepackage{amsthm}
\usepackage{mathrsfs}
\usepackage{hyperref}
\usepackage[english]{babel}
\usepackage{ wasysym }
\usepackage{enumitem}
\usepackage{amsfonts}
\theoremstyle{plain}

\newtheorem{theorem}{Theorem}
\newtheorem{lemme}{Lemma}
\newtheorem{proposition}{Proposition}
\theoremstyle{definition}
\newtheorem{definition}{Definition}
\theoremstyle{remark}

\newtheorem{remarque}{Remark}
\newcommand{\red}{\color{OrangeRed}}

\newcommand{\proba}[1]{\mathcal{P}_{\delta_N,T_N} \left( #1 \right) }
\newcommand{\probaPolymere}[1]{\mathbf{P}_{N,\delta_N}^{T_N} \left( #1 \right) }
\newcommand{\probaPolymereDeltaInvariant}[1]{\mathbf{P}_{N,\delta}^{T_N} \left( #1 \right) }
\newcommand{\probaPolymereSansN}[1]{\mathbf{P}_{N,\delta}^{T} \left( #1 \right) }

\newcommand{\probaPolymereSansParentheseniN}[0]{\mathbf{P}_{N,\delta}^T}
\newcommand{\probaPolymereSansParenthese}[0]{\mathbf{P}_{N,\delta_N}^{T_N}}
\newcommand{\probaPolymereSansParentheseDeltaInvariant}[0]{\mathbf{P}_{N,\delta}^{T_N}}
\newcommand{\probaRenouvellementSansParentheseniN}[0]{\mathcal{P}_{\delta,T}}

\newcommand{\probaRenouvellementSansParenthese}[0]{\mathcal{P}_{T_N,\delta_N}}

\newcommand{\probaSansN}[1]{\mathcal{P}_{\delta,T} \left( #1 \right)}

\newcommand{\EsperanceRenouvellement}[1]{\mathcal{E}_{\delta_N, T_N} \left( #1 \right) }

\newcommand\blfootnote[1]{%
  \begingroup
  \renewcommand\thefootnote{}%
  \footnotetext{#1}%
  \setlength{\parindent}{0pt}%
  \addtocounter{footnote}{0}%
  \endgroup
}

\newcommand{\somme}[2]{\underset{#1}{\overset{#2}{\sum}}}

\newcommand{\limite}[2]{\underset{#1 \longrightarrow #2}{\lim}}

\definecolor{britishracinggreen}{rgb}{0.0, 0.26, 0.15}

\newcommand{\cst}[0]{\nu}

\newcommand{\C}[0]{M}

\newcommand{\Anu}[0]{\mathcal{A}}

\newcommand{\BnuM}[0]{\mathcal{B}}

\newcommand{\CM}[0]{\mathcal{C}}

\newcommand{\pastouche}[0]{\mathcal{D}}

\newcommand{\R}[0]{\mathbb{R}}

\newcommand{\N}[0]{\mathbb{N}}

\newcommand{\po}[0]{p}

\newcommand{\D}[2]{\mathcal{D}_{#1}^{#2}}

\title{Polymer in a multi-interface medium with~weak~repulsion}
\author{Elric Angot}
\AtEndDocument{\bigskip{\footnotesize%
  \textsc{Nantes Universit\'e, Laboratoire Jean Leray, 2, rue de la Houssini\`ere, 44322 Nantes cedex 3, France} \par  
  \textit{E-mail address}, \texttt{elric.angot@univ-nantes.fr} \par
  \textit{Orcid number}:  \texttt{0009-0004-1698-0266}
}}

\date{\today}
\makeindex

\begin{document}

\maketitle
\numberwithin{equation}{section}
\numberwithin{proposition}{section}
\numberwithin{lemme}{section}

\begin{abstract}
%Long linear polymers in a depinned interfaces environment have been studied for a long time, for instance in \cite{Caravenna2009depinning} when the temperature is constant. In this paper, we study an extension of this model by making the temperature and the distance between interfaces goes to infinity. We give a full phase diagram for this model, and show a new behaviour of the polymer in a particular set of parameters. Our key tool include new sharp results on the simple random walk evolving between interfaces.

Pinning phenomena for long linear polymers have been studied for a long time.  In 2009 Caravenna and Pétrélis~\cite{Caravenna2009depinning} investigated the effect of a periodic and repulsive multi-interface medium on a $(1+1)$-directed polymer model, when the distance between consecutive interfaces scales with the length of the polymer and with a constant temperature. In this paper, we extend that model  and consider \emph{weak repulsion}, by letting both the temperature and the distance between interfaces scale with the length of the polymer. We obtain a full diagram for this model, showing the behaviour of the polymer depending on the scaling exponents associated to the repulsion and the spacing parameters. When the repulsion is not too weak compared to the interface spacing, we obtain different regimes that extend those obtained by Caravenna and Pétrélis, and either finitely or infinitely many interfaces are visited. When the two exponents match we obtain a diffusive regime with a non-trivial and temperature-dependent diffusion constant. Our key tools include the renewal approach used in the original paper as well as new sharp results on the simple random walk evolving between interfaces.

\end{abstract}

\tableofcontents

\section{Introduction}

\subsection{Model and notation}
Constrained or interacting random walks may exhibit an unusually confined behaviour within a very small region of space, in contrast to the  usual diffusive scaling.  This phenomenon can be observed, for instance, with the collapse transition of  a polymer in a poor solvent \cite{Carmona2018}, the pinning of a polymer on a defect line \cite{giacomin2007random,Giacomin2006smoothing} or when random walks are confined by obstacles \cite{Sznitman1998}. These models, often inspired by fields like Biology, Chemistry, or Physics, pose intricate mathematical challenges and have remained a vibrant area of research.

\par  We consider a (1 + 1)-dimensional polymer interacting with infinitely many equally spaced repulsive interfaces. The possible configurations of the polymer are modeled by the trajectories of the simple random walk  $(i,S_i)_{i \geq 0}$, where $S_0 = 0$ and $\left(S_i - S_{i-1}\right)_{i \geq 1} $ is a sequence of  independent and identically distributed (i.i.d.) random variables uniformly distributed on $\{-1,1\}$. We denote by $P$ the law of the simple random walk, and by $E$ its associated expectation. The polymer receives an energetic penalty $\delta$ (with $\delta > 0 $) every time it touches one of the horizontal interfaces located at heights $\{kT : k \in \mathbb{Z} \}$.  We assume that $T \in 2\mathbb{N}$ for convenience. More precisely, the polymer interacts with the interfaces through the following Hamiltonian:
\begin{equation}
    H_{N,\delta}^{T} := \delta \somme{i=1}{N}\mathbf{1} \{ S_i \in T \mathbb{Z} \} = \delta \somme{k \in \mathbb{Z}}{} \somme{i=1}{N}\mathbf{1} \{ S_i = kT \},
\label{1.1}
\end{equation}
\noindent where $N \in \mathbb{N}$ is the number of monomers constituting the polymer. We introduce the corresponding polymer measure $\probaPolymereSansParentheseniN$: %
\begin{equation}
    \frac{d  \probaPolymereSansParentheseniN}{d P}(S) := \frac{\exp\Big(-H_{N,\delta}^{T}(S)\Big)}{Z_{N,\delta}^{T}},
\label{1.2}
\end{equation}

\noindent where the normalisation constant $Z_{N,\delta}^{T} = E[ 
\exp(-H_{N,\delta}^{T}(S))]$ is the \textit{partition function}.

\par  Caravenna and Pétrélis~\cite{Caravenna2009depinning} considered the case when the repulsion strength $\delta$ is fixed while the interface spacing $T_N$ is allowed to vary with the size of the polymer $N$.
Before summarizing their results, we first set some 
notations. For $(a_N)_{N \in \N}$ a positive sequence, we write $S_N \simeq a_N$ if for all 
$\varepsilon > 0$, there exists an $M>0$ such that $\probaPolymereDeltaInvariant{|S_N/a_N| 
\geq M} \leq \varepsilon$ (that is, $\frac{S_N}{a_N}$ sampled from 
$\probaPolymereSansParentheseDeltaInvariant$ is a tight sequence), and if there exist 
$\nu>0$ and $0<p<1$ such that $\probaPolymereDeltaInvariant{|S_N/a_N|\geq \nu} \geq p$. 
 Caravenna and Pétrélis proved the following result
\cite[Theorem 1]{Caravenna2009depinning}:  for every penalty $\delta>0$:
\begin{equation}
S_N \text{ under } \probaPolymereSansParentheseDeltaInvariant  \simeq
\begin{cases}
  \sqrt{N/T_N} & \text{ if } \frac{T_N^3}{N} = O(1)  \\
  \min \{T_N,\sqrt{N}\} & \text{ if } T_N^3 \geq N.
\end{cases}
\end{equation}
Let us give a heuristic explanation for these scalings. For a fixed $T \in 2\mathbb{N}$, the process $(S_n)_{n=0}^N$ sampled from $\probaPolymereSansParentheseniN$ behaves similarly as a time-homogenous Markov process. In fact, the random variable $\hat{\tau} := \inf \{ n > 0 : |S_n| = T \}$ that gives the time needed for the polymer to switch interface is almost a renewal process. {\red{ Moreover, $\hat{\tau}$ is of order $T^3$, because of the repulsion strength given by interfaces. Indeed, for the simple random walk (SRW), it takes around $T^2$ steps to switch interfaces, but here, due to compensation with the free energy, the number of steps required is bigger. For more details, see \cite[(2.3)]{Caravenna2009depinning}, where one can notice that the $T^3$ comes from the second term in the Taylor expansion of the free energy. Note that the prefactor before this $T^3$ term is a function of $\delta$, hinting  at a possible change of behaviour in our model when $\delta$ goes to $0$.}}
%Roughly speaking, consider the time it takes for the walk to touch an interface for the first time. For the SRW, the probability of this time being large shows exponential decay for times greater than $T^2$, matching the expected interface change time - cf \eqref{probabilité pour la marche aléatoire simple que tau 1 vaille n}. In \cite{Caravenna2009depinning}, however, the exponential decay begins at $T^3$, due to compensation between the free energy and the exponential decay for the SRW. For more details, see \cite[(2.3)]{Caravenna2009depinning}, where one can notice that the $T^3$ comes from the second term in the Taylor expansion of the free energy, and the $T^2$ term will be simplified with the exponential decay of the SRW. Note that the prefactor before this $T^3$ term is a function of $\delta$, hinting the change of behaviour in our model when $\delta$ goes to $0$}}.

{\red{Because $\hat{\tau} $ is of order $T^3$}}, the polymer will perform around $N/T^3$ changes of interface. Assuming these considerations are still true when $(T_N)$ varies with $N$, the polymer will perform approximately $u_N := N/T_N^3$ changes of interface. By symmetry, the probability of going to the upper interface being the same as the probability of going to the lower interface, we deduce that $S_N \simeq T_N \sqrt{u_N}$ if $u_N \rightarrow \infty$. On the other hand, when $u_N \rightarrow 0$, the polymer  does not see any interface other than the origin, hence $S_N \simeq T_N$, at least when $T_N = O(\sqrt{N})$. Finally, if $u_N$ is bounded, the polymer will visit a finite number of interfaces.

\par In this paper we extend the model by investigating the case where both the interface spacing $T = T_N$ and the repulsion strength $\delta = 
\delta_N$ are allowed to vary with the size $N$ of the polymer. We assume that the sequence $(\delta_N)$ is non-increasing (weak interaction model). More precisely, we are 
interested in the following questions: at which level will $\delta_N$  matter? Is the 
asymptotic behaviour of the polymer modified by the interplay between $\delta_N$ and 
$T_N$? If so, how? When does the last contact with an interface take place?

\par Finally, let us cite the recent papers \cite{PoisatSimenhaus2020, Poisat2022} which deal with the case of randomly located repulsive interfaces. The positions of the interfaces do not vary with $N$ and are sampled according to a renewal process with polynomial inter-arrival distribution. A dichotomy between a weak form and a strong form of localisation, according to the value of the renewal exponent, is therein established for the polymer.

\par Throughout this paper, we use the following notation:
\begin{itemize}
    \item  $\mathbb{N}\cup \{0\}$ and $\mathbb{N}$ are the sets of non-negative and positive integers, respectively.
    \item $c$ is a constant whose value is irrelevant and which may change from line to line.
    \item $u_n = o_n(v_n)$, or $u_n \ll v_n$  means that $\limite{n}{\infty} \frac{u_n}{v_n} = 0$. 
    %\item Same for $o_N(u_N)$. We will denote $g(\varepsilon) = o(f(\varepsilon))$ if $\limite{\varepsilon}{0}\frac{g(\varepsilon)}{f(\varepsilon)} = 0$.
    \item $a_n \sim_n b_n$  means that $\frac{a_n}{b_n} \longrightarrow 1$ as $n \longrightarrow \infty$. 
    \item $u_n = O_n(v_n)$  means that $\left(\frac{u_n}{v_n}\right)_{n \in \mathbb{N}}$ is bounded.
    %\item $u_N = O_n(v_N)$ if $\left(\frac{u_N}{v_N}\right)_{N \in \mathbb{N}}$ is bounded.
    %\item $P(.)$ is the measure of the simple random walk. 
    %Sometimes we will use another probability measure, the projection of $\mathcal{P}_{T_N,\delta_N}$ on its first coordinate. By abuse of notation, we will use the same notation for this measure: $\proba{\tau=n} := \proba{\tau=n, \varepsilon \in \{-1,0,1\}}$ - see \eqref{définition de la mesure du renouvellement}. 
    %\item $P_{N,\delta_N}^{T_N}$ i the polymer measure with $N$ monomers, defined in \eqref{1.2}.
    %\item $E(X)$ is the expected value of $X$ under the simple random walk measure.
    %\item $E_{\delta_N,T_N}(X)$ is the expected value of $X$ under the renewal measure.
    %\item $E_{N,\delta_N}^{T_N}(X)$ is the expected value of $X$ under the polymer measure.
    \item $f(n) \asymp_n g(n) $ means that there exists $C>0$ such that for all $ n \in \mathbb{N}$, $\frac{1}{C} g(n) \leq f(n) \leq Cg(n)$.
    \item {\red{For $(a_{n,k})_{n \in \N, k \in I_n}$ and $(b_n)_{n \in \N}$ two sequences of real numbers, we write $a_{n,k} \asymp_n b_n$ if there exists $C>0$ such that, for all $n \geq 1$ and $k \in I_n$, $\frac{1}{C} b_n \leq a_{n,k} \leq C b_n$. The set $I_n$ will never be explicitly mentioned, but its meaning should be clear  from the context.  }}
    %\item We use in the same way $f(N) \asymp_N g(N)$.
    %\item $u_N \ll v_N$ if $u_N = o(v_N)$.
\end{itemize}

\subsection{ Main results}
We consider $(T_N)_{N \in \mathbb{N}} \in (2 \mathbb{N})^{\mathbb{N}}$ an non-decreasing sequence. The repulsion intensity is given by $( \delta_N)_{N \geq 1}$, such that $\delta_N \to 0$ as $N\to\infty$. For simplicity, we consider the power case, hence we assume that 
\begin{equation} \label{Pomme 1.4}
T_N = N^a \qquad \text{and} \qquad \delta_N = \frac{\beta}{N^b},
\end{equation}
where $a,b,\beta>0$ are fixed. 
It turns out that the same heuristic as the one described in the previous section can be applied when $\delta_N \rightarrow 0$ as $N \rightarrow \infty$. The first thing one can note is that, if $\delta_N T_N \rightarrow 0$ as $N \rightarrow \infty$, then the penalty is comparatively too low and $S_N$ behaves as the simple random walk. If $\delta_N T_N \geq 1$ for $N$ large  then the time to switch interface (i.e. to visit a neighboring interface) behaves as $T_N^3 \delta_N$. Assuming that this is true uniformly in $N$, the number of interface switches is of order $v_N := {N}/{(T_N^3 \delta_N)}$. Hence, if $v_N \rightarrow \infty$ as $N \rightarrow \infty$, then $S_N$ should be of order $T_N \sqrt{v_N} = \sqrt{{N}/{(T_N \delta_N)}}$. Otherwise, if $v_N \rightarrow 0$ as $N \rightarrow \infty$, $S_N$ should be of order $T_N$. As we will see with Theorems \ref{main theorem} and \ref{Theoreme 2} below, this heuristic turns out to be true, and there is indeed a change of behaviour when the repulsion strength crosses the threshold $N/T_N^3$, which in terms of the exponents $a$ and $b$ leads to comparing $b$ with $3a -1$.
\begin{figure}[H]
    \centering
    \includegraphics[width=\textwidth]{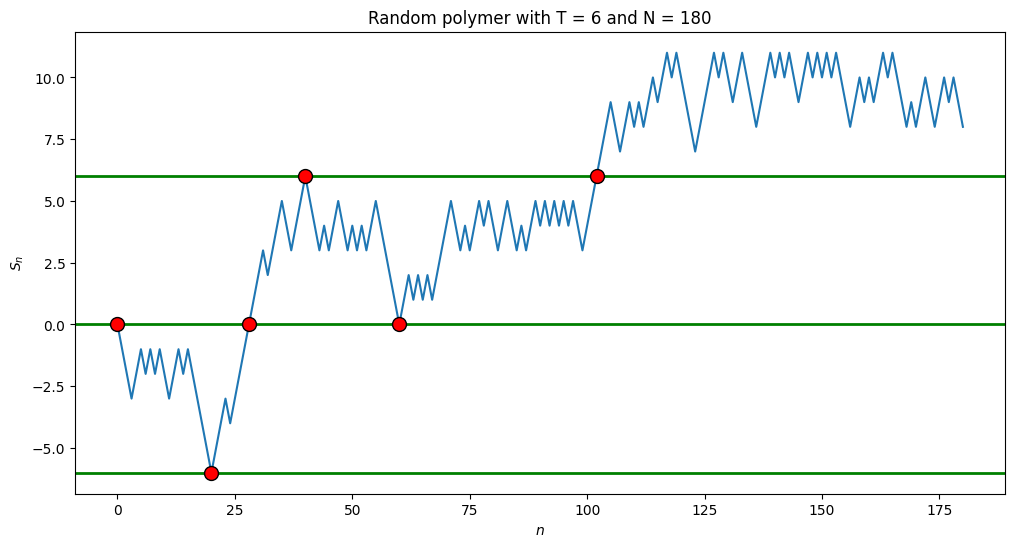}
    \caption{Example of a polymer trajectory of length 180, with $T=6$. The  contacts between the polymer and the repulsive interfaces are highlighted with red disks.} 
    \label{polymer exemple}
\end{figure}

Before stating our main results, we need the following:
\begin{definition} Let $(S_n)_{n \in \mathbb{N}}$ be the simple random walk on $\mathbb{Z}$. For every $T \in 2\mathbb{N} \cup \{ \infty \}$, we define $\left(\tau_k^T\right)_{k \in \mathbb{N}}$  as follows, with $\infty \, \mathbb{Z} := \{0\}$:

\begin{equation}
\tau_0^{T} = 0,  \qquad \tau_{k+1}^T = \inf \left\{n > \tau_k^T : S_n \in T \mathbb{Z} \right\}.
\label{4.5000}
\end{equation}
\noindent We also define $\tau^T :=   \{\tau_k^T : k \geq 0\}$ {\red{the set of contact points}} and
\begin{equation}
\varepsilon_k^T = \frac{S_{\tau_k^T} - S_{\tau_{k-1}^T} }{T}, \qquad k\in\mathbb{N}.
\end{equation}
Last, we define, for $N \in \mathbb{N}$ :
\begin{equation}
L_N := \max \left\{ n \in \mathbb{N}\cup \{0\} : \tau_n^T \leq N \right\}.
\end{equation}
  We might omit the superscript $T$ for convenience. Clearly, $\tau_{j}^{T}$ is the time of the $j$-th visit of the polymer to the interface set, while $\varepsilon_{j}^{T}$ indicates whether the
 $j$-th interface visited is the same as the
 $(j-1)$-th one $\left(\varepsilon_{j}^{T}=0\right)$, or the one above
 $\left(\varepsilon_{j}^{T}=1\right)$ or the one below $\left(\varepsilon_{j}^{T}=-1\right)$.
\label{definition 1}
\end{definition}
\par  We are now ready to state our results, which we split in two theorems. Theorem \ref{main theorem} is dedicated to the regime
\begin{equation}
    \begin{aligned}
        & 1/3 < a < 1/2, &\text{ that is} \quad & N^{1/3} \ll T_N \ll \sqrt{N}, \\
     \text{and}\quad   & 3a-1>b, &\text{ that is} \quad   &N \ll T_N^3 \delta_N,
    \end{aligned}
    \label{H_0}
\end{equation}
where the localisation of the last contact of the polymer is non trivial; while the other regimes are covered by Theorem~\ref{Theoreme 2}. The results of these two theorems are summarized in Figures \ref{fig: Phase Diagram} and \ref{fig:recap_tab}.

\begin{theorem}
{\red{Assume that either \( \frac{1}{3} < a < \frac{1}{2} \) and \( 3a - 1 > b \), or \( a < \frac{1}{2} \) and \( b \geq \frac{1}{2} \).}}
 For all $\varepsilon > 0$, there exist $\cst > 0$, $\C>0$ and $N_0 \in \mathbb{N}$ such that, for $N \geq N_0$:
\begin{equation}
\probaPolymere{{\frac{\cst}{ \delta_N^2} \leq \tau_{L_N}^{T_N} \leq  \frac{\C}{\delta_N^2}}} \geq 1-\varepsilon.
\label{ de terre}
\end{equation}
Moreover:
\begin{equation}
\limite{N}{\infty} \probaPolymere{\exists i \leq N : S_i \in T_N\mathbb{Z} \backslash \{0\}} =
0.
\label{Deuxieme partie theoreme}
\end{equation}
\label{main theorem}
\end{theorem}
In other words, in this regime the only interface visited by the polymer is the one at the origin and the last contact with it is roughly at $1/\delta_N^2 \ll N$, with high probability. The polymer eventually gets stuck between the origin and one of the two (above or below) neighbouring interfaces. {\red{The case $b \geq 1/2$ is less surprising, because the SRW does not reach points beyond $\sqrt{N}$, but we included it in Theorem \ref{main theorem} for completeness.}} {\red{We expect the endpoint to follow the quasi-stationary distribution for the walk constrained to stay in $[0,T_N]$, similarly as the law of $\mathcal{W}$  (center of localization) in \cite[Theorem 1.1]{bouchot2022scaling}. We will not prove this result here, as it would further extend an already lengthy paper.}}
{\red{
\begin{definition} In this section only, we say that $S_N$ is of order $v_N$ if there exists $C = C(\beta,a,b)>0$ such that, for all $u,v \in \bar \R$, there exists $N_0 \in \N$ such that, for all $N \geq  N_0$,
\begin{equation}
\probaPolymere{u  < \frac{S_N}{v_N} \leq v} \in P(u  < Z \leq v)[C^{-1},C],
\end{equation}
with $Z$ a standard normal distribution.
\end{definition}

}}
We now state our results on all non-border cases.
\begin{theorem}[{\red{Non-border cases}}] \phantom{ let us put some void there}
\begin{enumerate}
    \item  When $a>b>3a-1$, $S_N $ is of order $ 
\pi \sqrt{N/(T_N \delta_N)}$
    \label{partie 123.1 du théorème}  
%    When $a>b>3a-1$, $S_N $ is of order $ \sqrt{N/(T_N \delta_N)} $. More precisely, there exists $C>0$ such that, for all $u \leq v \in \overline{\R}$, there exists $N_0 \in \N$ such that, for $N \geq N_0$:
%\begin{equation}
%\probaPolymere{u  \leq  \frac{S_N}{\sqrt{\pi^2 N/(T_N \delta_N)  }} \leq v}  \in P(u  \leq  Z \leq v) \left[C^{-1},C\right].
%\label{123.1.266}
%\end{equation}

\item \label{partie 123.2 du théorème}  When $1/2>a$ and $b< a$, $S_N $ is of order $ \sqrt{N}$. 
%More precisely,  let $x_\beta$ be the only solution in $(0,\pi)$ of  $x_\beta = \frac{\sin(x_\beta)\beta}{1-\cos(x_\beta)}$, and $\kappa(\beta)$ be a constant defined as:
%
%\begin{equation}
%\kappa(\beta) =
%\begin{cases}
%%    1 & \text{if } b > a \\
 %   \sqrt{\frac{x_\beta^3}{\beta(x_\beta + %\sin(x_\beta))}} & \text{if } a=b.
%\end{cases}
%\end{equation}
%
%Then, there exists $C>0$ such that, for all $u \leq v \in \overline{\R}$, there exists $N_0 \in \N$ such that, for $N \geq N_0$:
%\begin{equation}
% \probaPolymere{u  \leq  \frac{S_N}{\kappa(\beta) \sqrt{N}} \leq v}  \in P(u  \leq  Z \leq v) \left[C^{-1},C\right].
%\label{123.1.366}
%\end{equation}
%More precisely, let $Z$ follow a standard normal distribution. There exists $C>0$ such that, for all $u \leq v \in \overline{\R}$, there exists $N_0 \in \N$ such that, for $N \geq N_0$:

%\begin{equation}
%\probaPolymere{ u \leq \frac{S_N}{\sqrt{N}} \leq v } \in P(u  \leq  Z \leq v) \left[C^{-1},C\right].
%\label{123.1.466}
%\end{equation}

\item \label{C'est la meme } When $b>1/2 $ and $ a \geq 1/2$, there exists $\varepsilon : \N \rightarrow \R_+$ such that $\varepsilon(N) \rightarrow 0 $ as $N \rightarrow \infty $ such that, for all $A \in  \sigma(S_1,...,S_N)$:
\begin{equation}
\probaPolymere{A} = P(A) + \varepsilon(N).
\label{123.1.56}
\end{equation}
%\item \label{partie 123.5 du theoreme} When \texorpdfstring{$b<1/2 $ and $ a \geq 1/2$}{}, for all $\varepsilon > 0$, there exist $\cst > 0$, $\C>0$ and $N_0 \in \mathbb{N}$ and such that for $N \geq N_0$:
%\begin{equation}
%\probaPolymere{{\frac{\cst}{ \delta_N^2} \leq \tau_{L_N}^{T_N} \leq  \frac{\C}{\delta_N^2}}} \geq 1-\varepsilon.
%\label{ de terre2}
%\end{equation}
%Moreover:
%\begin{equation}
%\limite{N}{\infty} \probaPolymere{\exists i \leq Nx    S_i \in T_N\mathbb{Z} \backslash \{0\}} =
%0.
%\label{Deuxieme partie theoreme2}
%\end{equation}

%
%Lastly, there exists $c>0$ (that depends of $\beta$) such that, for all $m\geq0$, there exists $N_0 \in \N$ such that, for all $N \geq N_0$ and : 
%\begin{equation}
%\probaPolymere{\# \{ i \le L_N : \varepsilon_i^{T_N} = \pm 1\} \geq m} \geq \Big(\frac{c}{m} \Big)^m
%\label{2.13}
%\end{equation}
%
\end{enumerate}
\label{Theoreme 2}
\end{theorem}
Last, we state a theorem about the three border cases:
\begin{theorem}[Border cases]  \phantom{ let us put some void there}
\begin{enumerate}
    \item \label{cas a=b<1/2} When $a=b$ and $b\leq 1/2$, $S_N$ is of order $\kappa(\beta) \sqrt{N}$, with $x_\beta$ being defined as the only solution in $(0,\pi)$ of  $x_\beta = \frac{\sin(x_\beta)\beta}{1-\cos(x_\beta)}$ and $\kappa(\beta) := \sqrt{\frac{x_\beta^3}{\beta(x_\beta + \sin(x_\beta))}}$.
\item When $b=1/2 $ and $ a > 1/2$, $S_N $ is of order $\sqrt{N}$. 
\item  \label{poisson de pomme de terre} {\red{When $b<1/2$ and $3a-1=b$, there exists  $\theta\in(0,1)$ and $c > 0$ such that, for all $m \in \N$, there exists $N_0 \in \N$ such that, for all $N \geq N_0$:}}

\begin{equation}
\left( \frac{c}{m}\right)^m \leq \probaPolymere{\# \{ i \le L_N : \varepsilon_i^{T_N} = \pm 1\} = m} \leq \theta^m.
\label{123.1.6p}
\end{equation}
Moreover, there exists $C>0$ and $N_0 \in \mathbb{N}$ such that for all $N \geq N_0$ and $\varepsilon > 0$:

\begin{equation}
\probaPolymere{\ \tau^{T_N} \cap [(1-\varepsilon)N,N] \neq \emptyset } \leq C \varepsilon.
\label{123.11.7}
\end{equation}
\end{enumerate}
\label{Pomme de terre Th critique}
\end{theorem}
Let us make some comments on Theorem \ref{Theoreme 2} and \ref{Pomme de terre Th critique}.

\begin{itemize}
    \item   The first two  items in Theorem \ref{Theoreme 2} and the first two  items in Theorem \ref{Pomme de terre Th critique} give us tightness of the appropriately renormalized endpoint distribution of the polymer, with a density that is comparable, in some sense, to the standard normal density, when $b> 3a-1$ and $a<1/2$ or $b=1/2$ and $a\geq 1/2$. Scaling is sub-diffusive if $a>b$, and diffusive otherwise.
    \item  Item \ref{C'est la meme } in Theorem \ref{Theoreme 2} is even more precise. It tells us that the polymer behaves as if there were  no constraint: {\red{the total variation distance between $P$ and $\probaPolymereSansParenthese$  goes to zero}}.
    \item  Item \ref{poisson de pomme de terre} in Theorem \ref{Pomme de terre Th critique} shows that the number of interfaces visited by the polymer in the border regime is of order one. Moreover, the probability to touch any interface in $[(1-\varepsilon)N,N)$ is small for the polymer, whereas it tends to one for the simple random walk, see Lemma \ref{lemme 123.6.6} for a precise statement.
    \item  We plot the value of $x_\beta$ and $\kappa(\beta)$ (defined in Theorem~\ref{Theoreme 2}, Item \ref{partie 123.2 du théorème}) as functions of $\beta$ in Figure~\ref{subfig:beta_image} and Figure~\ref{subfig:renormalization_image} respectively. 
    \item The result in~\eqref{123.1.6p}, that is in the border case, could be strengthened. Denoting $u_N$ the number of interface changes, it should be possible to prove, using the same ideas as \cite[Section 4]{Caravenna_2009}, that there exists a constant $C>0$ such that, for all $k \in \mathbb{N}\cup \{0\}$, $\frac{1}{C} \frac{\lambda^k}{k!} \leq  \probaPolymere{u_N = k} \leq C \frac{\lambda^k}{k!}$. {\red{Moreover, the upper bound in \eqref{123.1.6p} is actually weaker than what we actually prove in this paper: there exists $N_1 \in \N$ such that, for all $m \geq 0$ and $N \geq N_1$, $\proba{u_N = m} \leq \theta^m$.}}
    \end{itemize}

\begin{figure}[H]
    \begin{minipage}[t]{0.48\textwidth}
        \centering
        \includegraphics[width=\textwidth]{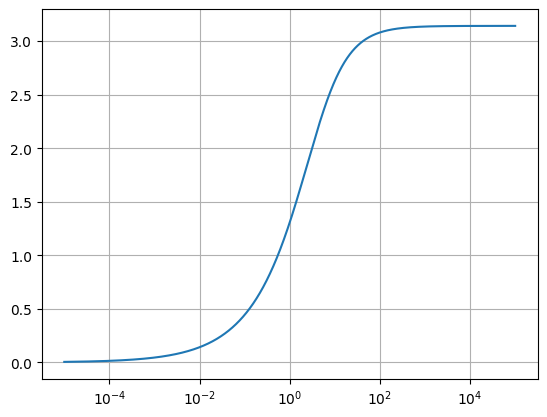}
        \caption{Plot of $\beta \mapsto x_\beta$.}
        \label{subfig:beta_image}
    \end{minipage}
    \hfill
    \begin{minipage}[t]{0.48\textwidth}
        \centering
        \includegraphics[width=\textwidth]{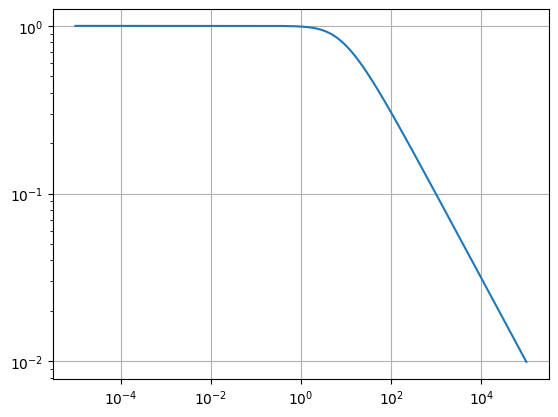}
        \caption{Plot of $\beta \mapsto \kappa(\beta)$, logarithmic scale.}
        \label{subfig:renormalization_image}
    \end{minipage}
    \label{fig:side_by_side}
\end{figure}

\begin{figure}[H]
  \centering
  \begin{tikzpicture}[scale = 1.8]

\draw[->] (0,0) -- (4.8,0);
\draw (4.8,0) node[right] {$a$};
\draw [->] (0,0) -- (0,4);
\draw (0,4) node[above] {$b$};
\draw (2,0) node[below] {$\frac{1}{3}$} ;
\draw (3,0) node[below] {$\frac{1}{2}$} ;
\draw (0,3) node[left] {$\frac{1}{2}$} ;
\draw (0, 0) -- (3, 3) ;
\draw (3, -0.05) -- (3, 0.05) ;
\draw ( -0.05,3) -- (0.05,3) ;
\draw (3,3) -- (3,4) ;
\draw (3,3) -- (4.8,3) ;
\draw (2,0) -- (3,3) ;
\draw[dotted] (3,0) -- (3,3) ;
\draw (1.3,1.6) node[scale=1,rotate=45]{$T_N = \delta_N$. Border case 1. Diffusive} ;
\draw (2.3,1.5) node[scale=1,rotate=67.5]{$T_N^3\delta_N = N$} ;
\draw (1.5,3) node[above] {Regime 2. Diffusive} ;
\draw (1,0) node[above] {Sub-Diffusive} ;
\draw (1,0.2) node[above] {Regime 1} ;
\draw (2.5,0) node[above] {Localized}  ;
\draw (2.55,0.2) node[above] {Theorem 1}  ;
\draw (4,0.8) node[above] {Theorem  1} ;
\draw (4,0.6) node[above] {Weakly Depinned } ;
\draw (4,3.5) node[above]{Regime 3 } ; 
\draw (4,3.3) node[above]{Simple Random Walk } ; 

\draw[->] (3.2,1.8) -- (2.6,1.8);
   
\node[anchor=west] at (3.2,1.8) {Border case 3. Finite};

\node[anchor=west] at (3.2,1.6){number of visited interfaces.};

\draw[->] (4,2.5) -- (4,3);

\node[anchor=north] at (4,2.5) {Border case 2. Diffusive};

\end{tikzpicture}
  \caption{ Exponent diagram. The labels of the different regimes match with the items in Theorem~\ref{Theoreme 2} and Theorem \ref{Pomme de terre Th critique}. }
  \label{fig: Phase Diagram}
\end{figure}

\begin{figure}[H]
    \centering
    \begin{tabular}{|c|c|c|c|c|c|c|c|}
        \hline
         Th 1, $b<1/2$ & Th1, $b \geq 1/2$ & R1 & R2 and BC1 & R3 and BC2  & BC3 \\
        \hline
        $T_N$ & $\sqrt{N}$ &  $\sqrt{\frac{N}{T_N \delta_N}}$ & $\sqrt{N}$ & $\sqrt{N}$ & $T_N$ \\
        \hline
         $\frac{1}{\delta_N^2}$ &  $\frac{1}{\delta_N^2}$ & $N$ & $N$ & $N$ & $N$ \\
        \hline
         $\frac{1}{\delta_N}$ & $\frac{1}{\delta_N}$ &$T_N^3\delta_N^2$ & $\min \left\{ \sqrt{N}, \frac{N}{T_N} \right\}$ & $\sqrt{N}$ &  $T_N^3 \delta_N^2$ \\
        \hline
    \end{tabular}
    \caption{Summary table. The second to fourth lines respectively contain, in order of magnitude, (i) the endpoint position of the polymer (ii) the localisation of the last contact of the polymer with the interfaces (iii) the number of contacts between the polymer and the interfaces.{\red{ The abbreviation $R$ stands for the regimes in Theorem \ref{Theoreme 2} and $BC$ for the border cases in Theorem \ref{Pomme de terre Th critique}.}}}
    \label{fig:recap_tab}
\end{figure}

{\red{We  illustrate below the qualitative behaviour of the polymer in different regimes. The horizontal lines represents  the interfaces while the  erratic line represents the polymer. }} 
\begin{figure}[H]
    \begin{minipage}[t]{0.45\textwidth}
        \centering
        \includegraphics[width=\textwidth]{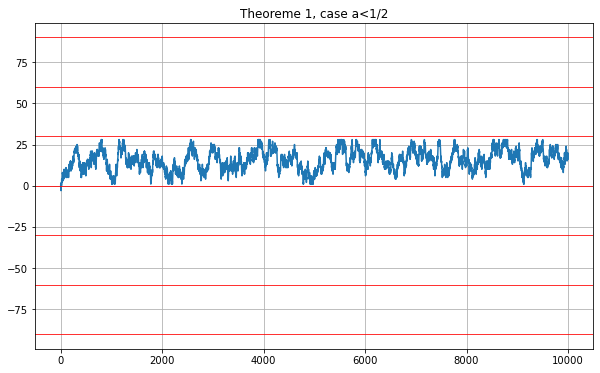}
        \caption{Theorem 1, case $a<1/2$. Note the last contact around $n=400$. }
        \label{fig:theorem1}
    \end{minipage}
    \hfill
    \begin{minipage}[t]{0.45\textwidth}
        \centering
        \includegraphics[width=\textwidth]{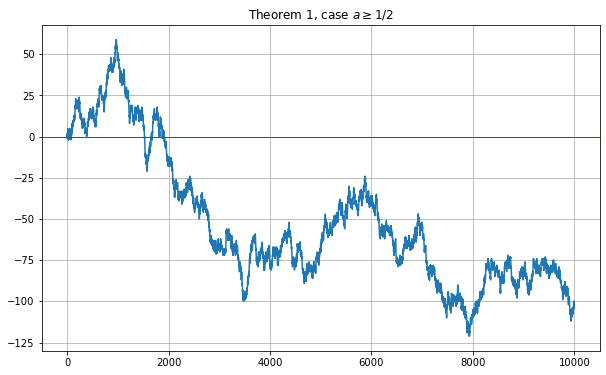}
        \caption{Theorem 1, case $a \geq 1/2$. Note the last contact around $n=2000$. Other interfaces do not appear as they are too spaced out.}
        \label{fig:regime2}
    \end{minipage}
    \begin{minipage}[t]{0.45\textwidth}
        \centering
        \includegraphics[width=\textwidth]{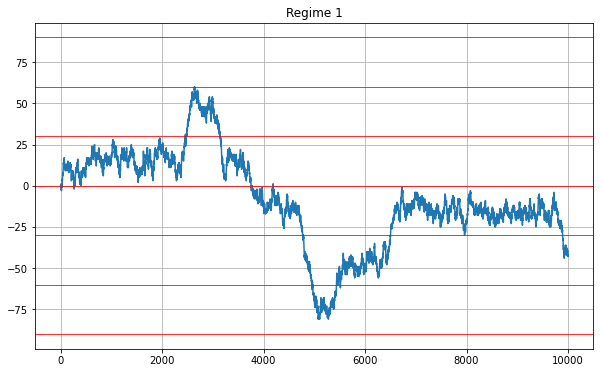}
        \caption{Regime 1. The polymer is sub-diffusive, but, as $N$ goes large, the number of interface changes  goes to infinity.} 
        \label{fig:regime6}
    \end{minipage}
    \hfill
    \begin{minipage}[t]{0.45\textwidth}
        \centering
        \includegraphics[width=\textwidth]{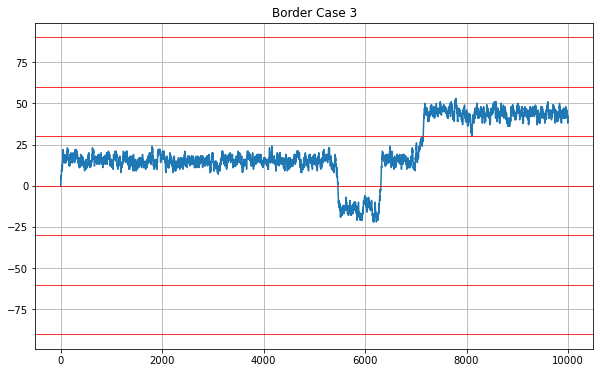}
        \caption{Border  case 3. The number of interface changes is  of order one and conjectured to follow a Poisson distribution.}
        \label{fig:regime1}
    \end{minipage}
    \label{fig:polymere}
\end{figure}

{\red{
\begin{remarque} All our results remain valid even outside the power case in \eqref{Pomme 1.4}. The specific changes are enumerated as follows:

\begin{enumerate}
    \item \textbf{Theorem \ref{Pomme de terre Th critique}, first border case:} the hypothesis $a=b$ becomes \(\delta_N T_N \longrightarrow c\).
    \item \textbf{Theorem \ref{Pomme de terre Th critique}, third border case:} the hypothesis $3a-b=1$ becomes \(\delta_N  T_N^3 /N  \longrightarrow c\).
    \item \textbf{Theorem \ref{Theoreme 2}, regime 1:} the hypothesis $a>b$ and $3a-b<1$ becomes \(\delta_N T_N \longrightarrow +\infty\) and \( \delta_N T_N^3/ N \longrightarrow 0 \).
    \item \textbf{Theorem \ref{Theoreme 2}, regime 2:} the hypothesis $a>b$ becomes \(\delta_N T_N \longrightarrow 0\).
    \item \textbf{Theorem \ref{main theorem}, case $3a-1+b>0$:} this hypothesis becomes \( T_N^3\delta_N/N \longrightarrow \infty\) .
    \item \textbf{Remaining cases:} The corresponding hypothesis becomes $T_N/\sqrt{N} \longrightarrow c$ or  $ +\infty$, and $\delta_N/\sqrt{N}$ tends to $0,c$ or $+\infty$ depending on the case.
\end{enumerate}
The proofs are essentially the same, but the computations involved are more technically challenging when we move beyond the power case. 

\end{remarque}
}}

%\vspace{0.5 cm}

As a byproduct of our analysis, we obtain results for the return times of the simple random walk to the interfaces, which may be of independent interest. We collect these results in Theorem \ref{theoreme} below. Because of our convention in Definition~\ref{definition 1}, we recall that $(\tau_k^\infty)$ is the sequence of return times to the origin.
\begin{theorem}
There exists $C>0$, $C'>0$ such that, for all $T \in 2\mathbb{N}$, $0<n<2 T^2$ and $k \in \mathbb{N}$:
\begin{equation}
P\left(\tau_k^T = n \right) \leq \frac{Ck}{\min\{T^3 , n^{3/2} \}} e^{-ng(T)} ,
\label{equation theoreme}
\end{equation}
with $g(T) = - \log \cos \left( \frac{\pi}{T}\right) = \frac{\pi^2}{2T^2 } + O_T\left( \frac{1}{T^4} \right)$. Moreover, for $n\geq 2 T^2$ and $k \in \mathbb{N}$:
\begin{equation}
P\left(\tau_k^T = n \text{ and } \tau_1^T \leq T^2,...,\tau_k^T- \tau_{k-1}^T \leq T^2 \right)  \leq \frac{C}{T^3}\left(1+\frac{C'}{T}\right)^ke^{-ng(T)}.
\label{8.6}
\end{equation}
\label{theoreme}
\end{theorem}
{\red{\begin{remarque} Although we believe a full {\red{Central Limit Theorem (CLT)}} should hold when the polymer is diffusive or sub-diffusive as described in Theorem \ref{Theoreme 2}, we do not  prove it here due to insufficiently precise estimates. Achieving a full CLT would require sharper estimates for the SRW in \eqref{equation theoreme} and \eqref{8.6}, which would then propagate through all our computations in this paper. For \( k=1 \), \cite[Lemma 4.2 and 4.3]{bouchot2022scaling} provide better approximations, but we are unable to reach the same level of precision for all \( k \in \mathbb{N} \).
\end{remarque}}}

{\red{ Another result proven in this paper is the following lemma:
\begin{lemme}
There exists $C>0$ such that, for all $k \geq 1$ and $n \in 2\mathbb{N}$:
\begin{equation}
P(\tau_k^\infty = n) \leq \frac{Ck}{n^{3/2}}. 
\label{2.3}
\end{equation}
\label{Lemme 1.1 Cantal}
\end{lemme}}}

Equation \eqref{2.3} may be obtained by combining \cite[Theorem 2.4]{berger2019strong} for $n \geq k^2$ 
and \cite[Theorem 1.1]{Doney2019} for $n \leq k^2$. We will present here a self-contained computational 
proof, which contains the ideas to prove \eqref{equation theoreme} and \eqref{8.6}. Moreover, note that \eqref{equation theoreme} 
is a refinement of \cite[Eq. 
(B.9)]{Caravenna2009depinning}.\\ 

\par Finally, we close this section by providing asymptotic behaviours of the partition function, defined below \eqref{1.2}, as these are often relevant to physicists. Before we state our result, let us slightly anticipate on the next section by defining the free energy as $ \phi(\delta,T) := \limite{n}{\infty} \frac{1}{n} \log Z_{n,\delta}^{T} $, see Section \ref{Section 123.3.1} for more details.
\begin{proposition}
The partition function defined in \eqref{1.2} satisfies:
\begin{equation}
Z_{N, \delta_N}^{T_{N}} \asymp_N \frac{1}{\max\{1, \delta_N \min \{ \sqrt{N}, T_N \}   \}} e^{N \phi(\delta_N,T_N)}    \quad \text{ if $a>b$,}
\label{123.1.66}
\end{equation}
\begin{equation}
Z_{N, \delta_N}^{T_{N}} \asymp_N e^{N \phi(\delta_N,T_N)}    \quad \text{ if  $a \leq b$. }
\label{123.1.77}
\end{equation}
\label{Proposition 123.1.3}
\end{proposition}
Proposition~\ref{Proposition 123.1.3} is proven in Appendix \ref{Annex F}.  In complement, we note that, when $a \geq 1/2 $ and $b \geq 1/2$,
\begin{equation}
Z_{N, \delta_N}^{T_{N}} \asymp_N 1.
\label{123.1.88}
\end{equation}
hinting at the fact that, in this case, the polymer does not feel the interaction with the interfaces.

\bigskip

\textbf{Outline of the paper.} We first introduce the necessary tools from renewal theory in Section \ref{Preliminaries}. This section ends with Proposition \ref{Prooposition 1} and Lemma  \ref{Lemme sur toucher en n pour la m.a.s.}, which are the two main ingredients to prove Theorem \ref{main theorem}. Sections \ref{Polymer measure} and \ref{Simple random walk}
contain the proofs of Theorems \ref{main theorem} and \ref{theoreme} respectively. They are then followed by Section \ref{section 60.0}, which contains the proof of Proposition \ref{Prooposition 1}. Section \ref{Section 123.77} contains all the tools and proofs of Theorems \ref{Theoreme 2} and \ref{Pomme de terre Th critique}. The general strategy of this paper is greatly inspired by \cite{Caravenna2009depinning}. It consists in writing the partition function in terms of a certain renewal process, after removal of the exponential leading term (see Section \ref{Preliminaries}). Part of the difficulty lies in the fact that, as in \cite{Caravenna2009depinning}, the renewal measure depends on the parameter $N$, hence non-asymptotic estimates on the renewal function are essential here. Some key estimates are however different from \cite{Caravenna2009depinning}, see e.g. Proposition \ref{Prooposition 1}, and therefore require additional work, making the results non-trivial.

\section{Preparation \label{Preliminaries}}

In this section, we first introduce the notion of free energy, a key concept in statistical mechanics.  Following~\cite{Caravenna2009depinning}, we then introduce a renewal process that is linked to the polymer  measure, that is another key concept coming from the study of pinning models, see~\cite{giacomin2007random, giacomin2011disorder}.

\subsection{Free energy and asymptotic estimates \label{Section 123.3.1}}

Let us assume for now that $T \in 2\mathbb{N}$ is fixed. We define the \textit{free energy} $ \phi(\delta,T) $ as the rate of exponential growth of the partition function defined below \eqref{1.2}:
\begin{equation}
    \phi(\delta,T) := \limite{N}{\infty} \frac{1}{N} \log Z_{N,\delta}^{T} 
    = \limite{N}{\infty}\frac{1}{N} \log E \left( e^{-H_{N,\delta}^{T}} \right).
\end{equation}
Investigating this function is motivated by identifying the points where the function $\delta \longrightarrow \phi(\delta, T)$ is not analytic. These points typically correspond to significant changes or transitions in the system, known as \textit{phase transition}. In our specific case, we do not observe any. However, even in the absence of phase transitions, studying this function remains valuable for understanding polymer trajectories in our model, as elaborated in Section \ref{théorie du renouvellement}. For this reason, we recall some basic results on free energy proven in \cite{Caravenna_2009}.

Let us remind Definition \ref{definition 1}:  $\tau_1^{T}$ is the random variable recording the time of the first contact of the simple random walk with an interface. We define $Q_{T}(\lambda) := E(e^{-\lambda \tau_1^{T}})$ its Laplace transform under the simple random walk law. There exists $\lambda_0^{T}$ such that $Q_{T}(\lambda)$ is finite and analytic on $(\lambda_0^{T}, \infty)$, where $Q_{T}(\lambda) \to +\infty$ as $\lambda \downarrow \lambda_0^{T}$. More precise expressions of $Q_{T}(\lambda)$ are known, see (A.4) and (A.5) in \cite{Caravenna_2009}. A common fact is that $Q_{T}(\cdot)$ is deeply connected to the free energy. More precisely:
\begin{equation}
    \phi(\delta, T)=\left(Q_{T}\right)^{-1}\left(e^{\delta}\right),
    \label{2.2}
\end{equation}
for all $\delta \in \mathbb{R}$ (this is Theorem 1 of \cite{Caravenna_2009}, with a change of notations since we have reversed the sign of $\delta$ compared to \cite{Caravenna_2009}). Let us come back to our model. Equation \eqref{2.2} gives us access to the asymptotic behaviour of $\phi(\delta_N, T_N)$ when $N \rightarrow \infty$, namely:
\begin{equation}
\phi (\delta_N, T_N) =   - \frac{\delta_N}{T_N}(1 + o_{N}(1)) \text{ if $a<b$},
\label{3.3000}
\end{equation}
\begin{equation}
\phi (\delta_N, T_N) =   - \frac{x_\beta^2}{2 T_N^2}(1 + o_{N}(1)) \text{ if $a=b$ }, 
\label{3.4000}
\end{equation}
\begin{equation}
\phi (\delta_N, T_N) =  - \frac{\pi^2}{2 T_N^2}
  \left(1 - \frac{4}{T_N \delta_N}
  (1 + o_{N}(1))
  \right)\text{ if $a>b$ },
\label{phi}
\end{equation}
where $x_\beta$ is an explicit constant defined in Theorem~\ref{Theoreme 2} and plotted in Figure~\ref{subfig:beta_image}. These estimates are proven in Appendix \ref{Annex A.1}. 

\subsection{ Renewal representation \label{théorie du renouvellement}}

We now rewrite our model in terms of a certain renewal process, following~\cite[Section 2.2]{Caravenna_2009}. Recall Definition \ref{definition 1}. We denote by
 $q_{T}^{j}(n)$ the joint law of $\left(\tau_{1}^{T}, \varepsilon_{1}^{T}\right)$ under the law of the simple random walk. It is defined for $(n,j) \in \mathbb{N} \times \{-1,0,1\}$:
\begin{equation}
q_{T}^{j}(n):=P\left(\tau_{1}^{T}=n, \varepsilon_{1}^{T}=j\right).
\end{equation}
Of course, by symmetry, $q_{T}^{1}(n)=q_{T}^{-1}(n)$ for all $n$ and $T$. We also define
\begin{equation}
q_{T}(n):=P\left(\tau_{1}^{T}=n\right)=q_{T}^{0}(n)+2 q_{T}^{1}(n).
\end{equation}
We now introduce $\probaRenouvellementSansParentheseniN$, the law of a Markov chain $\left\{\left(\tau_{j}, \varepsilon_{j}\right)\right\}_{j \geq 0}$ taking values in $(\mathbb{N} \cup\{0\}) \times$ $\{-1,0,1\}$ and defined as follows: $\tau_{0}=\varepsilon_{0}=0$ and, under  $\probaRenouvellementSansParentheseniN$, the sequence of random vector $\left\{\left(\tau_{j}-\tau_{j-1}, \varepsilon_{j}\right)\right\}_{j \geq 1}$ is i.i.d. with joint law:
\begin{equation}
\probaSansN{\tau_{1}=n, \varepsilon_{1}=j}:=e^{-\delta} q_{T}^{j}(n) e^{-\phi(\delta, T) n}.
\label{2.4}
\end{equation}
By \eqref{2.2}, this defines a probability measure. {\red{Indeed, a simple computation gives:
\begin{equation}
\somme{n \in \N}{} \somme{j \in \{-1,0,1\}}{} e^{-\delta} q_T^j(n) e^{-\phi(\delta,T)} = e^{-\delta }\somme{n \in \N}{} q_T(n) e^{-\phi(\delta,T)} = e^{-\delta} Q_T(\phi(\delta,T)) = 1.
\end{equation} 

}} 
\noindent Note that the process $\left(\tau_{j}\right)_{j \geq 0}$ alone under $\probaRenouvellementSansParentheseniN$ is a renewal process, for which $\tau_{0}=0$ and the variables $\left(\tau_{j}-\tau_{j-1}\right)_{j \geq 1}$ are i.i.d., with law

\begin{equation}
\probaSansN{\tau_{1}=n}=e^{-\delta} q_{T}(n) e^{-\phi(\delta, T) n}=e^{-\delta} P\left(\tau_{1}^{T}=n\right) e^{-\phi(\delta, T) n}.
\label{définition de la mesure du renouvellement}
\end{equation}

\noindent We also introduce a random set containing the indices at which the process touches one of the interface:
\begin{equation}
    \tau = \{\tau_i : i \in \mathbb{N} \cup \{0\}\},
\end{equation}
{\red {i.e., with a slight abuse of notation, we use the same letter $\tau$ to denote the renewal {\it process} and the corresponding {\it set} of renewal times. We shall use the same abuse of notation to the process $\tau^T$ defined below \eqref{4.5000}.
}} Let us now make the link between the law $\probaRenouvellementSansParentheseniN$ and $\probaPolymereSansParentheseniN$, the polymer measure defined in \eqref{1.2}. Recall the definition of $L_N$ in Definition \ref{definition 1}. By abuse of notation, we use $L_N$ in the context of the renewal process with the same definition, i.e., under $\probaRenouvellementSansParentheseniN $, $L_N := \max \left\{ n \in \mathbb{N} : \tau_n \leq N \right\} $. We now have the following crucial result (see \cite[Eq.  (2.13) ]{Caravenna_2009}): for all $N, T \in 2 \mathbb{N}$ and for all $k \in \mathbb{N},\left\{t_{i}\right\}_{1 \leq i \leq k} \in \mathbb{N}^{k},\left(\sigma_{i}\right)_{1 \leq i \leq k} \in\{-1,0,1\}^{k}$:
\begin{equation}
\begin{aligned}
&\probaPolymereSansN{L_{N}=k,\left(\tau_{i}^{T}, \varepsilon_{i}^{T}\right)=\left(t_{i}, \sigma_{i}\right), 1 \leq i \leq k \mid N \in \tau^{T}} \\
&=\probaSansN{L_{N}=k,\left(\tau_{i}, \varepsilon_{i}\right)=\left(t_{i}, \sigma_{i}\right), 1 \leq i \leq k \mid N \in \tau} ,
\end{aligned}
\label{3.12}
\end{equation}
where $\{N \in \tau\}:=\bigcup_{k=0}^{\infty}\left\{\tau_{k}=N\right\}$ and similarly for $\left\{N \in \tau^{T}\right\}$. In other words, the process $\left\{\left(\tau_{j}^{T}, \varepsilon_{j}^{T}\right)\right\}_{j}$ under $\probaPolymereSansN{\cdot \mid N \in \tau^{T}}$ is distributed in the same way as the Markov chain $\left\{\left(\tau_{j}, \varepsilon_{j}\right)\right\}_{j}$ under the measure $\probaSansN{\cdot \mid N \in \tau}$. It is precisely the link with the renewal theory that allows us to have precise estimates in our model. We shall also use on multiple occasions that
\begin{equation}
E\left[e^{-H_{k, \delta}^{T}(S)} \mathbf{1}_{\left\{k \in \tau^{T}\right\}}\right]=e^{\phi(\delta, T) k} \probaSansN{k \in \tau} ,
\label{3.13}
\end{equation}
which is true for all $k, T \in 2 \mathbb{N}$, see \cite[Eq. (2.11)]{Caravenna_2009}.
\subsection{Two renewal results}
We will use the following proposition to prove Theorem \ref{main theorem}:
\begin{proposition}
When $a>b$, i.e. when $T_N \delta_N  \longrightarrow 0$, there exists $T_0,c_1,c_2$, which may depend on $a$ and $b$, such that, for all $T_N \geq T_0$ and $n \in \mathbb{N}$:
\begin{equation}
\begin{aligned}
\frac{c_1}{\min \{ n^{3/2} \max\{ \frac{1}{n},\delta_N^2 \}, T_N^3\delta_N^2 \} } 
&\leq \proba{n \in \tau} \\
&\leq 
\frac{c_2}{\min \{ n^{3/2} \max\{ \frac{1}{n},\delta_N^2 \}, T_N^3\delta_N^2 \} }
\label{Proposition 1}.
\end{aligned}
\end{equation}
\label{Prooposition 1}
\end{proposition}
In other words, we have three regimes:
\begin{equation}
\proba{n \in \tau} \asymp_n
\left\{
\begin{array}{ll}
    \frac{1}{\sqrt{n}} & \text{if } n  \leq  \frac{1}{\delta_N^2}, \\
    \\
    \frac{1}{\delta_N^2 n^{3/2}} & \text{if } \frac{1}{\delta_N^2}  \leq  n  \leq  T_N^2,\\
    \\
    \frac{1}{T_N^3 \delta_N^2} & \text{if } n \geq T_N^2 \textrm{ (stationary regime)}.
\end{array}
\right.
\label{Pomme 2.15}
\end{equation}
Proposition \ref{Prooposition 1} extends \cite[Proposition 2.3]{Caravenna2009depinning} to the case $b>0$ and provides key estimates on the renewal process that is linked to the polymer measure via~\eqref{3.12} and \eqref{3.13}. {\red{When the penalty is constant ($b=0$), the polymer makes a big excursion (that is, when the time between two  consecutive contacts  with the interfaces is greater than $T_N^2$) after a time of order one, which takes it further than $N$. In our case, it makes a big excursion every $1/\delta_N$  contacts with the interfaces. Thus, it behaves like the simple random walk during its $1/\delta_N$ first contacts, then it goes on a long excursion taking it further than $N$.}} 

\vspace{0.5 cm}

\noindent  {\red{Another important result, as stated in \cite[Lemma 2.1]{Caravenna2009depinning} and refined in \cite[Lemma 4.2 and 4.3]{bouchot2022scaling}, is summarized in Lemma \ref{Lemme sur toucher en n pour la m.a.s.} below. Note that applying this refinement did not yield better estimates with our method, hence we write the sub-optimal lemma below. First, let us set}}
\begin{equation}
g(T):=-\log \cos \left(\frac{\pi}{T}\right)=\frac{\pi^{2}}{2 T^{2}}+O_{T}\left(\frac{1}{T^{4}}\right).
\label{g}
\end{equation}
We then have the following
\begin{lemme}
There exist $T_{0} \in 2\mathbb{N}, c_{2} > c_{1} > 0, c_{4} > c_{3} 
> 0$ such that, when $T \geq T_{0}$ is even and for $n \in 2 \mathbb{N}$:

\begin{equation}
\frac{c_{1}}{\min \left\{T^{3}, n^{3 / 2}\right\}} e^{-g(T) n} \leq P\left(\tau_{1}^{T}=n\right) \leq \frac{c_{2}}{\min \left\{T^{3}, n^{3 / 2}\right\}} e^{-g(T) n}, 
\label{probabilité pour la marche aléatoire simple que tau 1 vaille n}
\end{equation}
\begin{equation}
\frac{c_{3}}{\min \{T, \sqrt{n}\}} e^{-g(T) n} \leq P\left(\tau_{1}^{T} \geq n\right) \leq \frac{c_{4}}{\min \{T, \sqrt{n}\}} e^{-g(T) n}.
\label{probabilité pour la marche aléatoire simple que tau 1 soit plus grande que n}
\end{equation}

\label{Lemme sur toucher en n pour la m.a.s.}
\end{lemme}

\section{ Proof of theorem \ref{main theorem}: the localized regime \label{Polymer measure}}

In this section, we admit Proposition \ref{Prooposition 1} and prove Theorem \ref{main theorem} for the case $3a-1 > b$ and $1/3<a < 1/2$. The case $b \geq 1/2$ will be done at the end of this section. Sections~\ref{Section 4.1} and~\ref{sec:confinement} contain the proofs of the first and second statements in Theorem \ref{main theorem}, respectively.

\subsection{ Last contact with an interface
\label{Section 4.1}
}
\noindent We prove \eqref{ de terre}  in this subsection. For $\nu,M>0$, let us first introduce the following events:
\begin{equation}
\Anu_\cst = \left\{\tau_{L_N}^{T_N} < \frac{\cst}{\delta_N^2} \right\}, \quad
\BnuM_{\cst,\C} = \left\{ \frac{\cst}{\delta_N^2} \leq \tau_{L_N}^{T_N} \leq\frac{\C}{\delta_N^2} \right\}, \quad 
\CM_{\C} = \left\{ \frac{\C}{\delta_N^2} <  \tau_{L_N}^{T_N}  \right\}.
\label{44.2}
\end{equation}
We shall prove that there exists $C>0$ {\red{independent of $\nu$ and $M$}} such that:
\begin{equation}
\probaPolymere{\Anu_\cst} \leq C \sqrt{\cst} \probaPolymere{\BnuM_{\cst,\C}}
\label{Heroes 1}
\end{equation}
\begin{equation}
\probaPolymere{\CM_{\C}} \leq \frac{C}{\sqrt{\C}} \probaPolymere{\BnuM_{\cst,\C}},
\label{Heroes 2}
\end{equation}
\noindent which is equivalent to \eqref{ de terre} in Theorem \ref{main theorem}. In this section we abbreviate $\phi := \phi(\delta_N,T_N)$. Let us start with \eqref{Heroes 1}:
\begin{equation}
\begin{aligned}
&\probaPolymere{\Anu_\cst} = \frac{1}{Z_{N,\delta_N}^{T_N}}\somme{k=1}{\frac{\cst}{\delta_N^2}}E\left( e^{-H_{k,\delta_N}^{T_N}} \mathbf{1} \left\{ k \in \tau^{T_N} \right\} \right) P \left( \tau_1^{T_N} \geq N-k  \right) \\
&=\frac{1}{Z_{N,\delta_N}^{T_N}e^{-\phi N}}\somme{k=1}{\frac{\cst}{\delta_N^2}}e^{-k\phi}E\left( e^{-H_{k,\delta_N}^{T_N}}
1
\left\{ k \in \tau^{T_N} \right\}\right) P \left( \tau_1^{T_N} \geq N-k  \right) e^{-(N-k)\phi}. 
\label{44.4}
\end{aligned}
\end{equation}
First, we note that $P \left( \tau_1^{T_N} \geq N-k  \right) e^{-(N-k)\phi}  \asymp_N \frac{1}{T_N}$ because {\red{ $N-k \geq T_N^2$ $(a<1/2)$ and, by \eqref{probabilité pour la marche aléatoire simple que tau 1 soit plus grande que n}, $P \left( \tau_1^{T_N} \geq N-k  \right) e^{-(N-k)\phi} \asymp_N \frac{1}{T_N} e^{(-\phi - g(T_N))(N-k)}$ }. Moreover, \eqref{phi} and \eqref{g} give that $-(g(T_N) + \phi(\delta_N,T_N)) =  -\frac{2 \pi ^2}{T_N^3 \delta_N}(1 + o_N(1)) = o_N(1/N)$ $(3a-b>1)$, hence the first equality}. By using \eqref{3.13}:

\begin{equation}
\probaPolymere{\Anu_\cst} \asymp_N \frac{1}{T_N Z_{N,\delta_N}^{T_N} e^{-N \phi}}\somme{k=1}{\frac{\cst}{\delta_N^2}} \proba {k \in \tau}.
\label{4.565}
\end{equation}
With the upper bound in \eqref{Proposition 1}:
\begin{equation}
\somme{k=1}{\frac{\cst}{\delta_N^2}} \proba {k \in \tau} \asymp_N \somme{k=1}{\frac{\cst}{\delta_N^2}} \frac{1 }{\sqrt{k}}
\asymp_N \frac{\sqrt{\cst}}{\delta_N}.
\end{equation}
Therefore:
\begin{equation}
\probaPolymere{\Anu_\cst} \asymp_N \frac{\sqrt{\cst}}{T_N \delta_N Z_{N,\delta_N}^{T_N}e^{-N\phi}}. 
\label{9.7}
\end{equation}
With the same ideas as above, using the lower bounds of \eqref{Proposition 1} and \eqref{probabilité pour la marche aléatoire simple que tau 1 soit plus grande que n}:
\begin{equation}
\probaPolymere{\BnuM_{\cst,\C}} \asymp_N \frac{1}{T_N \delta_N Z_{N,\delta_N}^{T_N}e^{-N\phi}} \left( 1-\sqrt{\cst} + 1 - \frac{1}{\sqrt{\C}} \right).
\label{9.8}
\end{equation}
We now prove \eqref{Heroes 2} with the same ideas. We remark that at the first step of the proof, the pre-factor from \eqref{probabilité pour la marche aléatoire simple que tau 1 soit plus grande que n}  has a different behaviour when $N-k \leq T_N^2$, leading us to:
\begin{equation}
\probaPolymere{\CM_{\C}} \asymp_N
 \frac{1}{T_N\delta_N Z_{N,\delta_N}^{T_N}e^{-N\phi}} \left( \frac{1}{\sqrt{\C}} + \frac{N}{ T_N^3 \delta_N } +  \frac{1}{ \delta_N T_N} \right).
\label{9.9}
\end{equation}
Finally, we get the result by combining \eqref{9.7}, \eqref{9.8} and \eqref{9.9}, and noting that the last two terms in \eqref{9.9} are $o_N(1)$, because $3a-b>1$.

\subsection{ Confinement of the polymer}~\label{sec:confinement}

This section is devoted to the proof of \eqref{Deuxieme partie theoreme}. First, let us note that, up to a negligible term, we may restrict to polymer configurations whose last contact with an interface happens before  $\C/{\delta_N^2}$, thanks to \eqref{ de terre}. Let us denote $\pastouche = \{ \forall i \leq N, S_i \notin T_N\mathbb{Z} \backslash \{0\} \}$ the complement of the event in \eqref{Deuxieme partie theoreme}. By using the reflection principle~\cite[Theorem 11, chapter III]{Petrov}\footnote{Note that this reference contains a typo: the equation in Theorem 11 should have an inferior sign instead of a superior sign.} in the second line below, we obtain that, for any $\varepsilon > 0$, there exists $M>0$ such that: 
\begin{equation}
\begin{aligned}
\probaPolymere{\pastouche} &\leq \proba{\CM^c} + \frac{2}{Z_{N,\delta_N}^{T_N}}P\Big(\underset{1 \leq l \leq \frac{\C}{\delta_N^2}}{\max}S_l \geq T_N\Big) \leq \varepsilon + \frac{4}{Z_{N,\delta_N}^{T_N}}  P\Big(S_{\frac{\C}{\delta_N^2}} \geq T_N \Big).
\end{aligned}
\label{4.120100}
\end{equation}
To bound $P\Big(S_{\frac{\C}{\delta_N^2}} \geq T_N \Big)$ from above, we use a standard estimate~\cite[Proposition 2.1.2]{lawler2010random}: there exist $m \geq 0$ and $\lambda > 0$ such that, for all $A>0$  and $B>0$:
\begin{equation}
P\left(  S_B \geq A \right) \leq m \exp \left( - \lambda \frac{A^2}{B} \right).
\label{Markov exponentiel}
\end{equation}
Here, by choosing $A = T_N$ and $B = \frac{\C}{\delta_N^2}$: 
\begin{equation}
P\left(  S_{\frac{\C}{\delta_N^2}} \geq T_N \right) \leq m \exp 
\left( 
-\frac{\lambda T_N^2 \delta_N^2}{\C}
\right).
\end{equation}
Let us come back to \eqref{4.120100}. By \eqref{123.1.66} and \eqref{H_0}, there exists $C > 0$ such that $Z_{N,\delta_N}^{T_N} \geq \frac{C^{-1}}{\delta_N T_N}e^{N \phi(\delta_N,T_N)}$. Then, by using  \eqref{phi}, and recalling that $\delta_N T_N \to + \infty$ and $\delta_N T_N^3 \geq N$, we obtain:
\begin{equation}
\begin{aligned}
\probaPolymere{\pastouche} &\leq \varepsilon + o_{T_N}(1) + C \delta_N T_N \exp \left( -\frac{\lambda T_N^2 \delta_N^2}{\C} + \frac{ \pi^2 N}{ 2T_N^2}(1+o_{T_N}(1)) \right) = \varepsilon + o_N(1).
\end{aligned}
\end{equation}
\noindent Thus, \eqref{Deuxieme partie theoreme} is proven.

\medskip

The proof for the case $b \geq 1/2$ and $a<1/2$ is the exact same as the one done here, except that $P(\tau_1^{T_N}  \geq  N-k) \asymp_N {1}/{\sqrt{N-k}}$ instead of ${1}/{T_N}$ in \eqref{44.4}, hence the {term} $T_N$ in the denominator of \eqref{9.7} and \eqref{9.8} becomes $\sqrt{N}$, and \eqref{9.9} becomes $\probaPolymere{\CM_{\C}} \asymp_N { [\sqrt{MN} \delta_N Z_{N,\delta_N}^{T_N}e^{-N\phi}]^{-1}}$.

\section{ Simple random walk estimates for the return time to interfaces \label{Simple random walk} }

In this section, we prove Lemma \ref{Lemme 1.1 Cantal} and Theorem \ref{theoreme}. Section \ref{section 2} contains preliminary technical results. Section \ref{section 5.2} contains the proof of \eqref{2.3}, that bounds from above the number of contacts between the SRW and the origin before time $n$.  Section \ref{section 5.3} contains the proof of \eqref{equation theoreme}, that bounds from above the same  quantity when the SRW evolves between interfaces  for a time $n$ that is not too large. Finally, Section \ref{Section 5.4} contains the proof of \eqref{8.6}, that is an estimate of the number of contacts between the SRW and the interfaces before time $n$, when $n$ is large (i.e., $n \geq 2T^2)$ and when the contact process only performs "small excursions". The proof of these equations relies on induction. 
\subsection{Technical results \label{section 2}}
The proofs of the following lemmas are deferred to Appendix \ref{Annexe A}.
\begin{lemme}
For the simple random walk, when $n \in 2 \mathbb{N}$:
\begin{equation}
    P\left(n \in \tau^{\infty} \right) = \frac{\sqrt{2}}{\sqrt{\pi n}}\left( 1 - \frac{1}{4n} + O_n\left( \frac{1}{n^2} \right) \right),
\label{1.1000}
\end{equation}
\begin{equation}
    P\left(\tau_1^{\infty} = n\right) = \frac{\sqrt{2}}{\sqrt{\pi}n^{3/2}}\left( 1 + \frac{3}{4n} + O_n\left( \frac{1}{n^2} \right) \right).
    \label{1.2000}
\end{equation}
\label{lemme 1.1}
\end{lemme}
Finally, we have the following expansion of the cumulative distribution function of the return time to the origin:
\begin{lemme}
For $l \in \mathbb{N}$:
\begin{equation}
P \left( \tau_1^{\infty} \leq  2l \right) \leq 1 - \frac{1}{ \sqrt{\pi l}} - \frac{3}{8 \sqrt{\pi} \, l^{3/2}} + o_l \left( \frac{1}{l^{3/2}} \right).
\label{1.4}
\end{equation}
\label{lemme 1.3}
\end{lemme}
We close this technical section with an estimate of an integral appearing in \eqref{8.19} below:
\begin{lemme}
For $\varepsilon > 0$:
\begin{equation}
    \int_\frac{1}{2}^{1-\varepsilon} \frac{dt}{t^{3/2} (1-t)^{3/2}} = 
     \frac{2- 3 \varepsilon + O_{\varepsilon}(\varepsilon)}{\sqrt{\varepsilon}}.
    \label{intégrale}
    \end{equation}
\label{lemme 1.2}
\end{lemme}
\subsection{ Estimates of the time to hit the origin \label{section 5.2} }
{\red{Recall that the aim of this section is to prove Theorem \ref{theoreme}. In this subsection, we start with an inductive proof of the known result \eqref{2.3}, which we then adapt to establish the theorem}}. Let $Q(k)$ be our induction hypothesis:
\begin{equation}
\forall n \in \mathbb{N}, \qquad P(\tau_k^\infty = n) \leq \frac{C k}{n^{3/2}}.
\label{hypothese2}
\end{equation}
We actually want to show that there exists $C>0$ such that, if $Q(k)$ holds, then $Q(k+1)$ holds as well. Let us remark that $Q(1)$ holds for the choice $C = \sqrt{2}$: using \cite[Eq. (3.7)]{kesten2008introduction}, it comes: $ 
P(\tau_1^{\infty} = 2n) = \frac{1}{2n-1} P(\tau_1=n) = \binom{2n}{n}\frac{1}{(2n-1)2^{2n}}$. Let us now assume that $Q(k)$ holds for $k \in \mathbb{N}$ and prove below that $Q(k+1)$ holds. For $l \in \mathbb{N}$, we hereafter write $K_l(n) := P(\tau_l^\infty = n)$.\\

\noindent We start by a simple computation:
\begin{equation}
K_{k+1}(n) = \somme{l=\frac{n}{2}}{n-2}K_k(l) K_1(n-l) + \somme{l=\frac{n}{2}}{n}K_1(l) K_k(n-l).
\label{3.1}
\end{equation}
Using $Q(1)$, we obtain
\begin{equation}
\begin{aligned}
\somme{l=\frac{n}{2}}{n}K_1(l) K_k(n-l) \leq  \somme{l=\frac{n}{2}}{n} \frac{\sqrt{2}}{l^{3/2}} K_k(n-l) \leq \somme{l=\frac{n}{2}}{n} \frac{2^{5/2}}{n^{3/2}} K_k(n-l)  &\leq \frac{8}{n^{3/2}} P\left(\frac{n}{2} \leq \tau_k^\infty\right) \\
&\leq \frac{8}{n^{3/2}}.
\end{aligned}
\end{equation}
Hence, \eqref{3.1} can be bounded from above by:
\begin{equation}
K_{k+1}(n) \leq \somme{l=\frac{n}{2}}{n-2}K_k(l) K_1(n-l) + \frac{8}{n^{3/2}}.
\label{3.3}
\end{equation}
Let us give a sharp upper bound on the sum above. To do this, we split the sum in two parts:
\begin{equation}
\Sigma_1^n := \somme{l=\frac{n}{2}}{n-2n^\alpha-2}K_k(l) K_1(n-l),
\end{equation}
\begin{equation}
\Sigma_2^n := \somme{l=n-2n^\alpha}{n-2}K_k(l) K_1(n-l),
\label{3.5-}
\end{equation}
where $0<\alpha < 1/2$. We start by a lemma dealing with the first sum:
\begin{lemme}
Assume $Q(k)$. For $n$ large enough,
\begin{equation}
\Sigma_1^n  \leq  
\frac{kC}{\sqrt{\pi} n^{3/2}}\left( \frac{1}{n^{\alpha/2}} + \frac{1}{6n^{3\alpha/2}} + o_n \left( \frac{1}{n^{3\alpha/2}} \right) + O_n \left( \frac{1}{n^{1-\alpha/2}} \right) \right).
\end{equation}
\label{lemma 3.1}
\end{lemme}
\begin{proof} We use that, as a consequence of \eqref{1.2000},
\begin{equation}
 K_1(2p) \leq \frac{\sqrt{2}}{\sqrt{\pi}(2p)^{3/2}} + \frac{\sqrt{2}}{\sqrt{\pi} (2p)^{5/2} }   
\end{equation}
for $p$ large enough. We remark that we take only the even terms of the sum, because $\tau_1^\infty$ can take only even values. We also use \eqref{hypothese2}. Hence, it comes that:
\begin{equation}
\begin{aligned}
&\Sigma_1 = 
\somme{p=\frac{n}{4}}{n/2-n^\alpha - 1}K_1(n-2p) K_k(2p) \\
&\leq \somme{p=\frac{n}{4}}{n/2-n^\alpha - 1}
\frac{\sqrt{2}}{(n-2p)^{3/2}} \frac{Ck  } {\sqrt{\pi} (2p)^{3/2}} +  \somme{p=\frac{n}{4}}{n/2-n^\alpha - 1}
\frac{\sqrt{2}}{(n-2p)^{5/2}} \frac{Ck  } {\sqrt{\pi} (2p)^{3/2}} \\
&\leq \frac{\sqrt{2} Ck}{\sqrt{\pi}}
\int_{n/4}^{n/2 - n^\alpha} \frac{dt}{(2t)^{3/2} (1-2t)^{3/2}}
+
\frac{\sqrt{2} Ck}{\sqrt{\pi}}
\int_{n/4}^{n/2 - n^\alpha} \frac{dt}{(2t)^{3/2} (1-2t)^{5/2}}
\\
&\leq  \frac{ Ck}{\sqrt{2} \sqrt{\pi} n^2 } 
\int_{1/2}^{1 - \frac{2}{n^{1 - \alpha}}} \frac{dt}{t^{3/2} (1-t)^{3/2}} +
\frac{ Ck}{\sqrt{2} \sqrt{\pi} n^3 } 
\int_{1/2}^{1 - \frac{2}{n^{1 - \alpha}}} \frac{dt}{t^{3/2} (1-t)^{5/2}}
.
\end{aligned}
\label{8.19}
\end{equation}
Using Lemma \ref{lemme 1.2} with $\varepsilon = \frac{2}{n^{1 - \alpha}}$:
\begin{equation}
\int_{1/2}^{1 - \frac{2}{n^{1 - \alpha}}} \frac{dt}{t^{3/2} (1-t)^{3/2}} \leq 
\sqrt{2} \frac{\sqrt{n}}{n^{\alpha/2}} \left( 1 + O_n \left( \frac{1}{n^{1 - \alpha}} \right)\right).
\end{equation}
Moreover, $\int_{1/2}^{1 - \frac{2}{n^{1 - \alpha}}} \frac{dt}{t^{3/2} (1-t)^{5/2}} \sim \frac{n^{3/2}}{3 \sqrt{2} n^{3\alpha/2}}$. Therefore, \eqref{8.19} becomes:
\begin{equation}
\begin{aligned}
\Sigma_1^n  &
& \leq 
\frac{kC}{\sqrt{\pi} n^{3/2}}\left( \frac{1}{n^{\alpha/2}} + \frac{1}{6n^{3\alpha/2}} + o_n\left( \frac{1}{n^{3\alpha/2}} \right) +O_n \left( \frac{1}{n^{1-\alpha/2}} \right) \right).
\end{aligned}
\end{equation}
\end{proof}

\par Now, let us prove Lemma \ref{lemma 3.2}, which helps us deal with the sum in \eqref{3.5-}:
\begin{lemme}
We have the following bound from above: 
\begin{equation}
\Sigma_2^n \leq 
\frac{Ck}{n^{3/2}} \left( 1 + \frac{3}{2 n^{1-\alpha}}\right) \left( 1 - \frac{1}{\sqrt{\pi}n^{\alpha/2}} - \frac{3}{8  \sqrt{\pi} \, n^{3\alpha/2}} + o_n \left( \frac{1}{n^{3 \alpha / 2}} \right) \right).
\end{equation}
\label{lemma 3.2}
\end{lemme}
\begin{proof} By using $Q(k)$:
\begin{equation}
\begin{aligned}
\Sigma_2^n &\leq 
\somme{p=n-2n^\alpha}{n-2} \frac{Ck}{p^{3/2}} K_1(n-p) \leq \frac{Ck}{(n-n^\alpha)^{3/2}}P(\tau_1^\infty  \leq  2n^\alpha). 
\end{aligned}
\end{equation}
Thanks to Lemma \ref{lemme 1.3}, we then have:
\begin{equation}
\begin{aligned}
\frac{Ck}{(n-n^\alpha)^{3/2}}P(\tau_1^\infty  \leq  2n^\alpha) \leq 
& \frac{Ck}{n^{3/2}}  \left( 1 + \frac{3}{2 n^{1 - \alpha}}\right) \\
&\left( 1 - \frac{1}{\sqrt{\pi}n^{\alpha/2}} - \frac{3}{8  \sqrt{\pi} \, n^{3\alpha/2}} + o_n \left( \frac{1}{n^{3\alpha/2}} \right) \right).
\end{aligned}
\end{equation}
\end{proof}
We now have the tools to prove \eqref{2.3}.
\begin{proof} Let us start from \eqref{3.3}. Thanks to Lemmas \ref{lemma 3.1} and \ref{lemma 3.2}, it comes that, for $n$ large enough:
\begin{equation}
\begin{aligned}
K_{k+1}(n) &\leq \frac{8}{n^{3/2}} + \frac{Ck}{n^{3/2}} \left( 1 + \frac{3}{2 n^{1 - \alpha}}\right) \left( 1 - \frac{1}{\sqrt{\pi}n^{\alpha/2}} - \frac{3}{8n^{3\alpha/2}} \right) \\ &+ 
\frac{kC}{\sqrt{\pi} n^{3/2}}\left( \frac{1}{n^{\alpha/2}} + \frac{1}{6 n^{3\alpha/2}} + O_n \left( \frac{1}{n^{1-\alpha/2}} \right) + o_n \left( \frac{1}{n^{3\alpha/2}} \right) \right) \\
& \leq \frac{8}{n^{3/2}} + \frac{Ck}{n^{3/2}} \left( 1 - \frac{5}{24n^{3 \alpha / 2}} + o_n \left( \frac{1}{n^{3 \alpha / 2}} \right) + 
O_n\left( \frac{1}{n^{1 - \alpha}}  \right) \right).
\label{3.15}
\end{aligned}
\end{equation}
Therefore, when $\frac{3\alpha}{2} <1 - \alpha$ and for $n$ large enough:
\begin{equation}
K_{k+1}(n) \leq \frac{8}{n^{3/2}} + \frac{Ck}{n^{3/2}} \left( 1 - \frac{5}{24n^{3 \alpha / 2}}(1+o_n(1)) \right).
\label{3.16}
\end{equation}
Hence, 
\begin{equation}
K_{k+1}(n) \leq \frac{8}{n^{3/2}} + \frac{Ck}{n^{3/2}}.
\end{equation}
We thus obtain the desired result for $n$ large enough, let us say larger than $M$. Hence, because $K_k(n) = 0$ if $n \leq k$, we can deduce that $Q(k) $ implies $Q(k+1)$ if $k \geq M$ and for $C \geq 8$. Therefore, we have to prove $Q(i)$ for $i \leq M$. Using \cite[Eq. (B.9)]{Caravenna2009depinning} with $T = \infty$, we obtain that $K_k(n) \leq \frac{ck^3}{n^{3/2}}$. By choosing $C = \max \{ cM^2,8 \}$, \eqref{2.3} is proven.

\end{proof}

\subsection{ Estimates of the time to hit the interfaces \label{section 5.3} }
We now want to prove \eqref{equation theoreme}. We see that, if $n<T$, $P \left( \tau_k^T=n \right) = P\left(\tau_k^\infty =n \right)$, because the simple random walk needs at least $T$ steps to reach another interface. Therefore, thanks to \eqref{2.3}, there is a constant $C>0$ such that, for all $k>0$ and $n<T$:
\begin{equation}
P\left(\tau_k^T=n \right) \leq \frac{Ck}{\min\{ n^{3/2},T^3 \}}e^{-ng(T)}.
\end{equation}
We may ignore the exponential term since $n \leq 2T^2$ and $g(T)\sim \pi^2/(2T)$. The technique we are going to use is the same as in Section \ref{section 5.2}. Thus, let us prove the following proposition:
\begin{proposition}
There exists a constant $C>0$ such that, for all $1 \leq n \leq 2T^2$ even:
\begin{equation}
\forall k \in  \mathbb{N}, \quad P\left(\tau_k^T = n\right) \leq \frac{Ck}{n^{3/2}}.
\end{equation}
\label{proposition 4.1}
\end{proposition}
Let us first note that such constant exists for $1 \leq n<T$. As in Section \ref{section 5.2}, we proceed by induction. Let $C>0$ and $Q(k)$ be the induction hypothesis: 
\begin{equation}
\forall \, 1<n<2T^2 \text{ even },\qquad P\left(\tau_k ^T= n\right) \leq \frac{C k}{n^{3/2}}.
\label{hypothese}
\end{equation}
For $k=1$, the assertion above is true with the constant $c_2$ from \eqref{probabilité pour la marche aléatoire simple que tau 1 vaille n}. Now, let us assume that $Q(k)$ is true and consider $T \leq n \leq 2T^2$. Using the notation: $K_k^T(n) := P\left(\tau_k^T = n\right)$, we have:
\begin{equation}
K_{k+1}^T(n) = \somme{l=\frac{n}{2}}{n-2}K_k^T(l) K_1^T(n-l) + \somme{l=\frac{n}{2}}{n}K_1^T(l) K_k^T(n-l).
\label{4.4}
\end{equation}
By~\eqref{probabilité pour la marche aléatoire simple que tau 1 vaille n} and a straightforward upper bound,
\begin{equation}
\begin{aligned}
\somme{l=\frac{n}{2}}{n}K_1^T(l) K_k^T(n-l) &\leq  \somme{l=\frac{n}{2}}{n} \frac{c_2}{l^{3/2}} K_k^T(n-l) \leq \somme{l=\frac{n}{2}}{n} \frac{2^{5/2}c_2}{n^{3/2}} K_k^T(n-l) \\  &\leq \frac{8c_2}{n^{3/2}} P\left(\frac{n}{2} \leq \tau_k^T\right) 
\leq \frac{8c_2}{n^{3/2}}.
\end{aligned}
\label{8.34}
\end{equation}
Hence, \eqref{4.4} can be bounded from above by:
\begin{equation}
K_{k+1}^T(n) \leq \somme{l=\frac{n}{2}}{n-2}K_k^T(l) K_1^T(n-l) + \frac{8c_2}{n^{3/2}}.
\label{4.6}
\end{equation}
We split the sum above in three parts in order to bound it from above:
\begin{equation}
\Sigma_3^n := \somme{l=\frac{n}{2}}{n-\sqrt{n}}K_k^T(l) K_1^T(n-l),    
\label{4.7000}
\end{equation}
\begin{equation}
\Sigma_4^n := \somme{l=n-\sqrt{n}}{n-2n^\alpha-2}K_k^T(l) K_1^T(n-l),
\label{4.8000}
\end{equation}
\begin{equation}
\Sigma_5^n := \somme{l=n-2n^\alpha}{n-2}K_k^T(l) K_1^T(n-l),
\label{4.9000}
\end{equation}
where $0<\alpha<1/2$. We bound $\Sigma_3^n$ from above with the help of \eqref{probabilité pour la marche aléatoire simple que tau 1 vaille n} and the induction hypothesis, following the same computations as in the proof of Lemma \ref{lemma 3.1}:
\begin{equation}
\Sigma_3^n   
\leq  
kCc_2 \somme{l=n/2}{n-\sqrt{n}} \frac{1}{(n-l)^{3/2}l^{3/2}}
\leq \frac{kCc_2}{\sqrt{2} n^{3/2}} \left( \frac{1}{n^{1/4}}(1+o_n(1)) \right).
\label{Bourgogne 25}
\end{equation}
Notice that we may replace $K_1^T(n-l)$ by $K_1(n-l)$ in \eqref{4.8000} and \eqref{4.9000} when $n-l \leq T$. Therefore, 
\begin{equation}
\Sigma_4^n = \somme{l=n-\sqrt{n}}{n-2n^\alpha-2}K_k^T(l) K_1(n-l)  \leq  
\somme{l=n/2}{n-2n^\alpha-2}K_k^T(l) K_1(n-l).
\label{4.11000}
\end{equation}
From there, the following upper bound arises by using the exact same technique as in Lemma \ref{lemma 3.1}:
\begin{equation}
\Sigma_4^n \leq
\frac{kC}{\sqrt{\pi} n^{3/2}}\left( \frac{1}{n^{\alpha/2}} + \frac{1}{6n^{3\alpha/2}} + o_n \left( \frac{1}{n^{3\alpha/2}} \right) O_n \left( \frac{1}{n^{1-\alpha/2}} \right) \right).
\label{4.12000}
\end{equation}
Also, by using the exact same technique as in Lemma \ref{lemma 3.2}, we can bound from above \eqref{4.9000}:
\begin{equation}
\Sigma_5^n \leq
\frac{Ck}{n^{3/2}} \left( 1 + \frac{3}{2 n^{1 - \alpha}}\right) \left( 1 - \frac{1}{\sqrt{\pi}n^{\alpha/2}} - \frac{3}{8 \sqrt{\pi} \, n^{3 \alpha / 2}} + o_n \left( \frac{1}{n^{3 \alpha / 2}} \right) \right).
\label{4.13000}
\end{equation}
By summing the upper bounds in \eqref{4.13000} and \eqref{4.12000}, and by choosing $\alpha$ small enough such that $ 3 \alpha/2<1 - \alpha$, we obtain the same result as in \eqref{3.15} and \eqref{3.16}:
\begin{equation}
\Sigma_3^n + \Sigma_4^n + \Sigma_5^n
\leq 
\frac{Ck}{n^{3/2}} \left( 1 - \frac{5}{24 n^{3 \alpha / 2}}(1+o_n(1)) \right).
\label{4.14000}
\end{equation}
By summing the upper bounds in \eqref{Bourgogne 25} and \eqref{4.14000}, taking $\alpha$ small enough such that $ 3 \alpha/2 < 1/4$, and taking $n$ large enough such that $|o_n(1)| \leq 1 $ in \eqref{Bourgogne 25} and \eqref{4.14000}:
\begin{equation}
\somme{l=\frac{n}{2}}{n-2}K_k^T(l) K_1^T(n-l) \leq \frac{Ck}{n^{3/2}}.
\label{4.15000}
\end{equation}
Therefore, by summing \eqref{4.15000} and \eqref{4.6}, provided $n$ is large enough (say larger than $M$):
\begin{equation}
K_{k+1}^T(n) \leq  \frac{Ck}{n^{3/2}} + \frac{8c_2}{n^{3/2}}.
\end{equation}
Thus, if $C \geq 8c_2$, we have:
\begin{equation}
K_{k+1}^T(n) \leq \frac{C(k+1)}{n^{3/2}}.
\end{equation}
It remains to prove the desired assertion for all $n \leq M$. Using \cite[Eq. (B.9)]{Caravenna2009depinning}: $K_k^T(n) \leq \frac{ck^3}{\min\{n^{3/2},T^3\}}$. Hence, by taking $C = \max \{ cM^2, 8c_2 \}$, we have proven the induction step. The proof of Proposition \ref{proposition 4.1} is now complete.
\subsection{ Estimates of the time to hit the interfaces with a small excursion constraint \label{Section 5.4}}
We move on to \eqref{8.6}. It turns out that our method fails for the regime $n \geq T^2$ because, unlike the previous case, a typical trajectory now has several large excursions. Reminding that the aim of Theorem \ref{theoreme} is to prove Proposition \ref{Prooposition 1}, we can consider a smaller set of random walk paths that will be sufficient. Let us note that, thanks to Proposition \ref{proposition 4.1}, there exists $C_0 > 0$ such that, when $n \leq 2T^2$ and $k \in \mathbb{N}$:
\begin{equation}
\begin{aligned}
P\left(\tau_k^T = n \text{ and } \tau_1^T \leq T^2,...,\tau_k^T- \tau_{k-1}^T \leq T^2 \right) 
& \leq P(\tau_k^T = n)\\
&\leq \frac{C_0}{T^3}\left(1+\frac{C'}{T}\right)^k e^{-ng(T)}.
\end{aligned}
\end{equation}
Let us now prove~\eqref{8.6}. We define:
\begin{equation}
Q_k^T(n) := P\left(\tau_k^T = n \text{ and } \tau_1 \leq T^2,...,\tau_k^T- \tau_{k-1}^T \leq T^2\right).
\label{4.20}
\end{equation}
Let us take $C>C_0$ a real number. Let $H(k)$ be the induction hypothesis:
\begin{equation}
\forall n \geq 2T^2, \qquad Q_k^T(n) \leq \frac{C}{T^3} \left(1+\frac{C}{T}\right)^ke^{-ng(T)}.
\end{equation}
Let us remark that $H(1)$ is true because $Q_1^T(n) = 0$ when $n \geq 2T^2$. We now prove that $H(k)$ implies $H(k+1)$. Let us assume that $H(k)$ is true. Let $n \geq 2T^2$ be an integer. Then,
\begin{equation}
\begin{aligned}
Q_{k+1}^T(n) &= \somme{l=1}{T^2} Q_1^T(l) Q_k^T(n-l). 
\end{aligned}
\label{4.21}
\end{equation}
In order to bound \eqref{4.21} from above, let us use that $g(T) \leq 10/T^2$, so that $e^{lg(T)} \leq 1 + \frac{C''l}{T^2}$ for $l \leq T^2$ and a certain $C''>0$. Hence, thanks to $H(k)$:
\begin{equation}
\begin{aligned}
&\somme{l=1}{T^2} Q_1^T(l) Q_k^T(n-l) \leq 
\frac{C e^{-ng(T)}}{T^3} \somme{l=1}{T^2} Q_1^T(l) e^{lg(T)} \left(1+\frac{C'}{T}\right)^k\\ &\leq \frac{C e^{-ng(T)}}{T^3} \left(1+\frac{C'}{T}\right)^k \somme{l=1}{T^2} Q_1^T(l) \left(1+\frac{C''l}{T^2} \right)\\
&\leq \frac{C e^{-ng(T)}}{T^3} \left(1+\frac{C'}{T}\right)^k \left( \proba{\tau_1^T \leq T^2} + \somme{l=1}{T^2} \frac{c_2}{l^{3/2}} \frac{C''l}{T^2} \right) \\
&\leq \frac{C e^{-ng(T)}}{T^3} \left(1+\frac{C'}{T}\right)^k \left( 1 + \frac{C'''}{T} \right),
\end{aligned}
\end{equation}
where we have used \eqref{probabilité pour la marche aléatoire simple que tau 1 vaille n} in the penultimate line. Therefore, the induction step is proven when we take $C'$ bigger than $C'''$, which completes the proof of \eqref{8.6}.

\section{Asymptotic renewal estimates \label{section 60.0}}

Our goal in this section is to prove Proposition \ref{Prooposition 1}. To this end, we first provide preliminary estimates in Section \ref{section 5000} and then prove Proposition \ref{Prooposition 1} in three steps: we first treat the case $n \geq T_N^3 \delta_N$ in Section \ref{section 6} and then proceed with the lower and upper bounds in Sections \ref{section 7} and \ref{section 8} respectively.

\subsection{ Preliminary estimates \label{section 5000}}

Let us consider the renewal process $(\tau_{n})_{n \geq 0}$ with law $\probaRenouvellementSansParenthese$. It turns out that the law of $\tau_{1}$ under $\probaRenouvellementSansParenthese$ has two sorts of excursions: the small ones, of order $O_N(1)$, carry a mass $1-\delta_N$ while the big ones, of order $O_N\left( T_N^{3} \delta_N\right)$, carry a mass $\delta_N$. Although we do not fully prove this statement, it is useful to keep it in mind for the sequel. We start with an estimate of the renewal function that is easily derived from Lemma \ref{Lemme sur toucher en n pour la m.a.s.}:
\begin{lemme}
There exist $c_2>c_1>0$ and $T_0 \in 2\mathbb{N}$ such that, for all $n \in 2\mathbb{N}$ and $T_N \geq T_0$:
\begin{equation}
\begin{aligned}
    \frac{c_1}{\min \{ n^{3/2}, T_N^3\}} e^{-(g(T_N) + \phi(T_N, \delta_N))n} &\leq \proba {\tau_1 = n} \\
    &\leq \frac{c_2}{\min \{ n^{3/2}, T_N^3 \} } e^{-(g(T_N) + \phi(T_N, \delta_N))n}.
    \label{probabilité que tau 1 = n}
\end{aligned}
\end{equation}
\end{lemme}
Let us stress that this is always true, even beyond the assumptions in \eqref{H_0}.
\begin{proof} Equation \eqref{probabilité que tau 1 = n} is an immediate consequence of \eqref{définition de la mesure du renouvellement} and \eqref{probabilité pour la marche aléatoire simple que tau 1 vaille n}.
\end{proof}
\begin{lemme}
Assume $a>b$. There exists $N_{0} \in \mathbb{N}$ and $c_{3} > 0$ such that, for $N>N_{0}$, the inequalities below are verified for any $m, n \in 2 \mathbb{N} \cup\{+\infty\}$ with $m<n$. 
\begin{equation}
\begin{aligned}
\proba{\tau_1 \geq m} \leq 
c_3 \left( \frac{\mathbf{1} \{ m \leq T_N^2 \} }{\sqrt{m}} + \delta_N \right) e^{-(g(T_N) + \phi(\delta_N,T_N))m}
\label{majoration de la proba que tau 1 soit > m},
\end{aligned}
\end{equation}
\begin{equation}
\proba {m \leq \tau_1 < n} \geq 
c_1 \delta_N \left( e^{-(g(T_N) + \phi(\delta_N,T_N))m} - e^{-(g(T_N) + \phi(\delta_N,T_N))n} \right),
\label{minoration de la proba que tau 1 soit entre m et n}
\end{equation}
\label{lemme 4.2}
where $c_1$ is the same constant as in \eqref{probabilité que tau 1 = n}.
\end{lemme}

\begin{proof} In order to prove \eqref{minoration de la proba que tau 1 soit entre m et n}, we sum over $k$ the lower bound in \eqref{probabilité que tau 1 = n}. We use that $\min \left\{T_N^{3}, k^{3 / 2}\right\} \leq T_N^{3}$ and we note that, by \eqref{phi} and \eqref{g}: 
\begin{equation}
g(T_N) + \phi(\delta_N,T_N) =  \frac{2 \pi ^2}{T_N^3 \delta_N}(1 + o_N(1)). 
\label{g + phi}
\end{equation}
\noindent To obtain \eqref{majoration de la proba que tau 1 soit > m}, we distinguish between two cases.
\begin{itemize}
    \item For $m \geq T_N^2$, we sum the upper bound in \eqref{probabilité que tau 1 = n} using \eqref{g + phi}. 
    \item For $m \leq T_N^2$, we use that $\frac{1}{\min \{ n^{3/2}, T_N^3 \}}  \leq  \frac{1}{T_N^3} + \frac{1}{n^{3/2}}$ and sum the upper bound in \eqref{probabilité que tau 1 = n}.
\end{itemize}
\end{proof}

We continue with the following:
\begin{lemme}\label{lem:5.3}
When $a>b$,
\begin{equation}
\EsperanceRenouvellement{\tau_1} = \frac{T_N^3 \delta_N^2}{2 \pi^2} (1 + o_N(1)),
\label{espérance de tau 1}
\end{equation}
\begin{equation}
\proba{\varepsilon_1^2 = 1} = \frac{\delta_N}{2}(1 + o_N(1)),
\label{probabilité de changer d'interface}
\end{equation}
\begin{equation}
\EsperanceRenouvellement{\tau_1^2} = \frac{T_N^6 \delta_N^3}{2 \pi^4}(1 + o_N(1)).
\label{moment d'ordre 2 de tau 1}
\end{equation}
\end{lemme}
The (inverse of the) expected value of $\tau_1$ gives us the limit of $\proba{n \in \tau}$ (stationary regime), by the Renewal Theorem. Equation \eqref{moment d'ordre 2 de tau 1} gives us access to the variance of this random variable. Those equalities are proven in Appendix \ref{Appendix A.2}. Heuristically, Equation \eqref{espérance de tau 1} combined with \eqref{probabilité de changer d'interface} tells us that, the average time to switch interface is $T_N^3\delta_N$, which is greater than $N$ under the assumption of \eqref{H_0}. Thus, we should not see any interface change before N in that regime.

\subsection{ Estimates of the renewal mass function: the long range \label{section 6}}
Let us assume for the moment that Proposition \ref{Prooposition 1} is true for $n \leq T_N^3 \delta_N$. Our aim in Section \ref{section 6} is to extend this proposition to $2\mathbb{N}$. 
\subsubsection{Lower bound}
We denote $\mu_m := \inf \{ k \geq m : k \in \tau \}$. For $n \geq T_N^3 \delta_N$: 
\begin{equation}
\begin{aligned}
&\proba{n \in \tau} \geq \proba{\tau \cap\left[n-T_N^3\delta_N, n-1\right] \neq \emptyset, n \in \tau} \\
&=\sum_{k=n-T_N^3\delta_N}^{n-1} \proba{\mu_{n-T_N^3\delta_N}=k} \proba{n-k \in \tau} \\
 &\geq 
\frac{c}{T_N^{3} \delta_N^2} \proba{\tau \cap\left[n-T_N^3\delta_N, n-1\right] \neq \emptyset}.
\end{aligned}
\label{Bourgogne 1}
\end{equation}

\noindent The factor $1/T_N^3\delta_N^2$ comes from Proposition \ref{Prooposition 1} admitted until $T_N^3\delta_N$. It is now sufficient to show that there exist a constant {\rm (cst)} and $T_{0}>0$ such that, for $T_N \geq T_{0}$ and $n \geq T_N^3 \delta_N$:
\begin{equation}
\proba{\tau \cap\left[n-T_N^3\delta_N, n-1\right] \neq \emptyset} \geq  \rm cst.
\label{Bourgogne 2}
\end{equation}
We prove the following equivalent fact:
\begin{equation}
\proba{\tau \cap\left[n-T_N^3\delta_N, n-1\right] \neq \emptyset} \geq C 
\proba{\tau \cap\left[n-T_N^3\delta_N, n-1\right]=\emptyset },
\label{Bourgogne 3}
\end{equation}
for an appropriate value of $C>0$. Thanks to \eqref{probabilité que tau 1 = n}:
\begin{equation}
\begin{aligned}
&\proba{\tau \cap\left[n-T_N^3\delta_N, n-1\right] \neq \emptyset }\\&
=\sum_{l=0}^{n-T_N^3\delta_N-1} \proba{l \in \tau} \sum_{k=n-T_N^3\delta_N}^{n-1} \proba{\tau_{1}=k-l} \\
&\geq c \sum_{l=0}^{n-T_N^3\delta_N-1} \proba{l \in \tau} \delta_N \Bigg(e^{-(\phi(\delta_N, T_N)+g(T_N))\left(n-T_N^3\delta_N-l\right)}  \\
&\hspace{5.2 cm} -e^{-(\phi(\delta_N, T_N)+g(T_N))(n-l)}\Bigg).
\end{aligned}
\label{Bourgogne 4}
\end{equation}
Thanks to \eqref{majoration de la proba que tau 1 soit > m}: 
\begin{equation}
\begin{aligned}
&\proba{\tau \cap\left[n-T_N^3\delta_N, n-1\right]=\emptyset} 
\\
&=\somme{l=0}{n-T_N^3\delta_N-1} \proba{l \in \tau} \somme{k \geq n}{} \proba{\tau_{1}=k-l} \\
& \leq c \somme{l=0}{n-T_N^3\delta_N-1} \proba{l \in \tau} e^{-(\phi(\delta_N, T_N)+g(T_N))(n-l)}  
\delta_N.
\end{aligned}
\label{Bourgogne 5}
\end{equation}
\noindent Thus, by \eqref{g + phi}:
\begin{equation}
\begin{aligned}
&\frac{ \left( e^{-(\phi(\delta_N, T_N)+g(T_N))\left(n-T_N^3\delta_N-l\right)}-e^{-(\phi(\delta_N, T_N)+g(T_N))(n-l)} \right) \delta_N  }{ \left( e^{-(\phi(\delta_N, T_N)+g(T_N))(n-l)} \right) \delta_N }  \\
&\geq  
e^{T_N^3\delta_N(\phi(\delta_N, T_N)+g(T_N))}-1 \sim_N e^{2 \pi^2} - 1,
\end{aligned}
\label{Bourgogne 6}
\end{equation}
thus this quantity is (eventually) bounded from below by a positive real number. We finally get the desired result.
\subsubsection{ Upper bound}
It remains to prove the upper bound of Proposition \ref{Prooposition 1} for $n \geq  T_N^{3} \delta_N$. It comes:
\begin{equation}
\begin{aligned}
&\proba{\tau \cap\left[n-T_N^{3}\delta_N, n-T_N^2 \right] \neq \emptyset, n \in \tau} \\
& =\sum_{k=n-T_N^{3} \delta_N}^{n-T_N^{2}} \proba{\mu_{n-T_N^{3}\delta_N}=k} \proba{n-k \in \tau} \\
& \leq \frac{c}{\delta_N^2 T_N^{3}} \proba{\tau \cap\left[n-T_N^{3}\delta_N, n-T_N^{2} \right] \neq \emptyset } \leq \frac{c}{\delta_N^2 T_N^{3}},
\end{aligned}
\label{7.1}
\end{equation}
by applying the upper bound of Proposition \ref{Prooposition 1} to $\proba{n-k \in \tau}$, because $T_N^{2} \leq n-k \leq T_N^{3} \delta_N$. If we now show that there is a constant {\rm (cst)} and $N_{0}>0$ such that, for $N \geq N_{0}$ and for $n \geq T_N^{3}\delta_N$,
\begin{equation}
\proba{\tau \cap\left[n-T_N^{3}\delta_N, n-T_N^{2}\right] \neq \emptyset \mid n \in \tau} \geq \rm cst,
\label{Bourgogne 10}
\end{equation}
we deduce that 
\begin{equation}
\proba{n \in \tau} \leq \frac{1}{ {\rm cst}} \proba{\tau \cap \left[n-T_N^{3}\delta_N, n-T_N^{2}  \right] \neq \emptyset , n \in \tau } \leq \frac{ {\rm cst}}{T_N^3 \delta_N^2},
\label{Bourgogne 11}
\end{equation}
and the proof is over. Instead of \eqref{Bourgogne 10}, we prove the following equivalent relation
\begin{equation}
\begin{aligned}
&\proba{\tau \cap\left[n-T_N^{3} \delta_N , n-T_N^{2}  \right] \neq \emptyset, n \in \tau } \\&\geq C \proba{\tau \cap \left[n - T_N^{3}\delta_N, n-T_N^{2}\right]=\emptyset, n \in \tau },
\end{aligned}
\label{Bourgogne 12}
\end{equation}
for some $C>0$. Let us start with the first term:
\begin{equation}
\begin{aligned}
&\proba{\tau \cap\left[n-T_N^{3}\delta_N, n-T_N^{2}\right] \neq \emptyset, n \in \tau} \\
&=\sum_{m=0}^{n-T_N^{3}\delta_N-1} \proba{m \in \tau} 
\sum_{l=n-T_N^{3}\delta_N}^{n-T_N^{2}} 
\proba{\tau_{1}=l-m} \proba{n-l \in \tau}.
\end{aligned}
\label{Bourgogne 13}
\end{equation}
Let us notice that $\proba{n-l \in \tau} \geq c / T_N^{3}\delta_N^2$ for $n-l \in 2 \mathbb{N}$ thanks to the lower bound of Proposition \ref{Prooposition 1}. By \eqref{minoration de la proba que tau 1 soit entre m et n} and \eqref{g + phi},
\begin{equation}
\begin{aligned}
&\sum_{l=n-T_N^{3}\delta_N}^{n-T_N^{2}} \proba{\tau_{1}=l-m} \\
& \geq c \delta_N \left(e^{-(\phi(\delta_N, T_N)+g(T_N))\left(n-T_N^{3}\delta_N-m\right)}-e^{-(\phi(\delta_N, T_N)+g(T_N))\left(n-T_N^{2}-m\right)}\right) \\
&=c \delta_N   e^{-(\phi(\delta_N, T_N)+g(T_N))\left(n-T_N^{3}\delta_N -m \right)}
\left(1-e^{-(\phi(\delta_N, T_N)+g(T_N))\left(T_N^{3}\delta_N-T_N^{2} \right)}\right) \\
&\geq c \delta_N  e^{-(\phi(\delta_N, T_N)+g(T_N))\left(n-T_N^{3} \delta_N -m\right)} 
\geq
c \delta_N e^{-(\phi(\delta_N, T_N)+g(T_N))(n-m)}.
\end{aligned}
\label{Bourgogne 14}
\end{equation}
Returning to \eqref{Bourgogne 13}:
\begin{equation}
\begin{aligned}
& \proba{ \tau \cap\left[n-T_N^{3}, n-T_N^{2}\right] \neq \emptyset, n \in \tau } 
\geq \\& \frac{c}{T_N^{3} \delta_N} \sum_{m=0}^{n-T_N^{3}-1} \proba{ m \in \tau} e^{-(\phi(\delta_N, T_N)+g(T_N))(n-m)}.
\end{aligned}
\label{Bourgogne 15}
\end{equation}
Now, let us focus on the second term of \eqref{Bourgogne 12}:
\begin{equation}
\begin{aligned}
&\proba{\tau \cap\left[n-T_N^{3}\delta_N, n-T_N^{2}\right]=\emptyset, n \in \tau} \\
&=\sum_{m=0}^{n-T_N^{3}\delta_N-1} \proba{m \in \tau} \sum_{l=n-T_N^{2}}^{n} \proba{\tau_{1}=l-m} \proba{n-l \in \tau}.
\end{aligned}
\label{Bourgogne 16}
\end{equation}
Since $l-m \geq T_N^{3}\delta_N-T_N^{2}$, thanks to the upper bound in \eqref{probabilité que tau 1 = n}, we obtain
\begin{equation}
\begin{aligned}
\proba{\tau_{1}=l-m} & \leq \frac{c}{T_N^{3}} e^{-(\phi(\delta_N, T_N)+g(T_N))(l-m)} 
& \leq \frac{c}{T_N^{3}} e^{-(\phi(\delta_N, T_N)+g(T_N))(n-m)},
\end{aligned}
\label{Bourgogne 18}
\end{equation}
because $n-l \leq  T_N^{2}$ (let us remember \eqref{g + phi} and  $T_N \delta_N \to \infty$). Moreover, thanks to the upper bound of Proposition \ref{Prooposition 1} applied to $\proba{n-l \in \tau}$, for $n-l \leq T_N^{2}$, we obtain
\begin{equation}
\sum_{l=n-T_N^{2}}^{n} \proba{n-l \in \tau} \leq c 
\somme{l=1}{ 1/\delta_N^2}\frac{1}{\sqrt{l}} + \frac{c}{\delta_N^2} \somme{ l= 1/\delta_N^2}{T_N^2 } \frac{1}{l^{3/2}} \leq \frac{c}{\delta_N}.
\label{Bourgogne 30}
\end{equation}
Returning to \eqref{Bourgogne 16}, we obtain 
\begin{equation}
\begin{aligned}
&\proba{\tau \cap\left[n-T_N^{3}\delta_N, n-T_N^{2}\right]=\emptyset, n \in \tau } \\
&\leq \frac{c}{ T_N^{3} \delta_N} \sum_{m=0}^{n-T_N^{3}\delta_N-1} \proba{m \in \tau} e^{-(\phi(\delta_N, T_N)+g(T_N))(n-m)}.
\end{aligned}
\label{Bourgogne 17}
\end{equation}
By comparing \eqref{Bourgogne 15} and \eqref{Bourgogne 17}, we see that \eqref{Bourgogne 12} is proven, which concludes the proof.
\subsection{Lower bound on the renewal mass function: the short range \label{section 7}}
In this section, we deal with the case $n \leq T_N^3 \delta_N$. We allow ourselves to neglect the exponential terms among equations in Lemma \ref{lemme 4.2}. Here is the main idea of the proof: we consider only those trajectories where only one excursion is larger than $cn$, with $c>0$ to be defined below, and where all other excursions are smaller than $cn$. We will see that such trajectories are indeed typical. Let us start with some estimates:

\begin{lemme}
There exist $c_5$ and $c_6$ such that, for all $1 \leq n \leq T_N^2$:
\begin{equation}
\EsperanceRenouvellement{\tau_1 \mathbf{1} \{ \tau_1 \leq n \}} \leq c_5 \sqrt{n},
\label{7.10}
\end{equation}
\begin{equation}
\EsperanceRenouvellement{\tau_1 | \tau_1 \leq n} \leq c_6 \sqrt{n}.
\label{7.11}
\end{equation}
\label{lemme 7.5}
\end{lemme}
\begin{proof} We use \eqref{probabilité que tau 1 = n} and standard computations for \eqref{7.10}. To prove \eqref{7.11}, we use that $\proba{\tau_1 \leq n} \geq \proba{\tau_1 = 2} \geq c$.
\end{proof}
Let us now turn to the proof of Proposition \ref{Prooposition 1} for $n \leq T_N^3 \delta_N$. Rather than computing $\proba{n \in \tau}$, we compute $\proba{(2 \frac{c_6}
{c_1}+1)n \in \tau}$, with $c_1$ appearing in \eqref{minoration de la proba que tau 1 soit entre m et n} and $c_6$ in \eqref{7.11}. Details of this choice, which helps us simplify computations, appear in Equation \eqref{7.111}.  Omitting the integer part, we then introduce
\begin{equation}
n' = \Big( 2\frac{c_6}{c_1} + 1 \Big)n.
\label{7.3}
\end{equation}
We define $\xi_i := \tau_i - \tau_{i-1}$, and $s_n := \min\{n, 1/\delta_N^2\}$. An excursion is said "small" if it is smaller than $s_n$. For $k$ in $\mathbb{N}$, we define the following event:
\begin{equation}
A_k(n) := \{  \xi_1 + ... + \xi_k = n' \text{ and }
\# \{ i \leq k \colon \xi_i  \leq  s_n \}= k-1 \},
\end{equation}
which represents the event of reaching $n'$ in $k$ excursions, when only one of them is larger than $s_n$. Thus, 

\begin{equation}
\proba{n' \in \tau} \geq \somme{k=1}{ \sqrt{s_n}/c_1} \proba{A_k(n)}
\end{equation}
since we only consider a subset of all trajectories reaching $n'$. We also define for $k$ in $\mathbb{N}$:
\begin{equation}
B_k(n) := \left\{  \xi_1 + ... + \xi_k  \leq  2 \frac{c_6}{c_1} n \right\} \cap \left \{ \forall i \leq k,  \xi_i \leq   s_n \right\},
\end{equation}
which represents the fact that small excursions do not exceed $2\frac{c_6}{c_1} n$ and the $k$ first excursions are smaller than $n$, and 
\begin{equation}
C_k(n) := \{\forall i \leq k, \, \xi_i  \leq  s_n \},
\end{equation}
which represents the probability of having $k$ small excursions. We now give an estimate of $\proba{A_k(n)}$:
\begin{lemme}
For $k \leq \frac{\sqrt{s_n}}{c_1}$ and $n \leq T_N^3 \delta_N$, 
\begin{equation}
    \proba{A_k(n)} \geq  \frac{ {\rm cst} k}{\min \{n^{3/2}, T_N^3\}} \left(1-\frac{c_1}{\sqrt{n}} \right)^{k-1}.
\end{equation}
\label{lemme 7.6}
\end{lemme}
\begin{proof}[Proof of Lemma \ref{lemme 7.6}] We first decompose the event on the location of the big excursion and use \eqref{probabilité que tau 1 = n}. Remind that we may neglect all exponential terms in these equations:
\begin{equation}
\begin{aligned}
& \proba{A_k(n)}  =\somme{j=1}{k} \somme{
\substack{l_1,...,l_{k} \in \mathbb{N} \\
l_1 + ... + l_{k} =n' \\
l_j \geq s_n \\
l_i  \leq  s_n \text{ for } i \neq j }}{}
\proba{(\xi_1,...,\xi_k) = (l_1,...,l_k)} \\
& = k \somme{
\substack{1 \leq l_1,...,l_{k-1}  \leq  s_n\\
l_1 + ... + l_{k-1} \leq \left( 2\frac{c_6}{c_1} \right) n \\
}}{}
\proba{(\xi_1,...,\xi_{k-1}) = (l_1,...,l_{k-1})}  \proba{\xi_k = n' - \somme{i=1}{k-1} l_i }  \\
& \geq \frac{k c_1}{\min \{ n^{3/2}, T_N^3 \}} \somme{
\substack{1 \leq l_1,...,l_{k-1}  \leq  s_n \\
l_1 + ... + l_{k-1} \leq \left( 2\frac{c_6}{c_1} \right) n \\
}}{}
\proba{(\xi_1,...,\xi_{k-1}) = (l_1,...,l_{k-1})} \\
& \geq 
\frac{k c_1}{\min \{ n^{3/2}, T_N^3 \}} \proba{B_{k-1}(n)}.
\end{aligned}
\label{7.18}
\end{equation}
We now are looking for a lower bound of $\proba{B_{k-1}(n)}$. This can be done by conditioning on the fact that all excursions are smaller than $s_n$, which corresponds to the event $C_{k-1}(n)$.
\begin{equation}
\begin{aligned}
&\proba{B_{k-1}(n)}\\ &= \proba{B_{k-1}(n) | C_{k-1}(n)} \proba{C_{k-1}(n)} \\
& \geq \proba{\xi_1 + ... + \xi_{k-1}  \leq  2\frac{c_6}{c_1}n \Big| \forall i <k, \, \xi_i \leq s_n } \proba{\tau_1 \leq s_n}^{k-1} \\
& \geq  \proba{\xi_1 + ... + \xi_{k-1}  \leq  2 k c_6 \sqrt{n} \Big|\forall i <k, \, \xi_i \leq s_n } \proba{\tau_1 \leq s_n}^{k-1}.
\end{aligned}
\end{equation}
where we used that $k \leq \frac{\sqrt{s_n}}{c_1}$. To make the following equation more readable, we note $X$ a random variable  that is distributed as $\xi_1 + ... + \xi_{k-1} $ conditioned on $\xi_i  \leq  s_n$ for all $i < k$. Lemma \ref{lemme 7.5} therefore gives us:
\begin{equation}
\begin{aligned}
\proba{B_{k-1}(n)} & \geq \proba{X \leq 2E(X)} \proba{\tau_1 \leq s_n}^{k-1} \\
& \geq \frac{1}{2}\left( 1- \frac{c_1}{\sqrt{s_n}} \right)^{k-1},
\end{aligned}
\label{7.19}
\end{equation}
where we used the Markov inequality and \eqref{minoration de la proba que tau 1 soit entre m et n} in the last line. Put together, \eqref{7.18} and \eqref{7.19}  yield Lemma \ref{lemme 7.6}. 
\end{proof}
We can now conclude the proof. By Lemma \ref{lemme 7.6}: 
\begin{equation}
\begin{aligned}
&\proba{n \in \tau} \geq \somme{k=1}{\sqrt{s_n}/c_1}\proba{A_k(n)}  
\geq \frac{c}{\min \{n^{3/2},T_N^3\}} \somme{k=1}{\sqrt{s_n}/c_1} k \left(1- \frac{c_1}{\sqrt{s_n}} \right)^{k-1} \\
& \geq \frac{c}{2 \min \{n^{3/2},T_N^3\}}  \left( 
- \frac{\sqrt{s_n}}{c_1} \frac{(1-\frac{c_1}{\sqrt{s_n}})^{\sqrt{s_n}/c_1}}
{\frac{c_1}{\sqrt{s_n}}} 
+ 
\frac{1 - (1 - \frac{c_1}{\sqrt{s_n}})^\frac{\sqrt{s_n}}{c_1}}{\left(\frac{c_1}{\sqrt{s_n}} \right)^2} \right) \\
&\geq \frac{c (1 - 2 e^{-1})}{\min \{ n^{3/2} \max\{ \frac{1}{n},\delta_N^2 \}, T_N^3\delta_N^2 \}},
\end{aligned}
\label{7.111}
\end{equation}
hence the desired result.
\subsection{ Upper bound on the renewal mass function: the short range \label{section 8}}

This section is the most technical one. In  Section \ref{Section 8.11}, we expose the proof for $n \leq $, by comparing our renewal measure to the simple random walk. Then, for $ 1/\delta_N^2 \leq n \leq \delta_N T_N^3$,  the proof becomes more involved as we need the sharp simple random walk estimates from Theorem \ref{theoreme}.
\subsubsection{Upper bound for the very short range \label{Section 8.11} }
We now work with $n \leq 1/\delta_N^2$. Let us start by a lemma about the simple random walk:
\begin{lemme}
There exists $C>0$ such that, for all $T \in 2\mathbb{N}$ and $n \in 2\mathbb{N}$:
\begin{equation}
    \frac{1}{C \min\{\sqrt{n}, T\}} \leq P(S_n \in T \mathbb{Z}) \leq \frac{C}{\min\{\sqrt{n}, T\}}.
\label{Bourgogne 22}
\end{equation}
\label{lemme 8.1}
\end{lemme}
We defer the proof to Appendix \ref{Annexe B.1}. We now prove the upper bound of Proposition \ref{Prooposition 1} for $n \leq  1/\delta_N^2$ by using the definition of $\probaRenouvellementSansParenthese$ in \eqref{définition de la mesure du renouvellement}.  Recall the definition of $\tau_k^{T_N}$ in \eqref{4.5000}.  By \eqref{probabilité pour la marche aléatoire simple que tau 1 vaille n} and Lemma \ref{lemme 8.1}:
\begin{equation}
\begin{aligned}
\proba{n \in \tau} &= \somme{k=1}{\infty}P \left( \tau_k^{T_N} = n \right) e^{-k\delta_N}e^{-n\phi(\delta_N,T_N)} 
&\leq c \somme{k=1}{\infty}P \left( \tau_k^{T_N} = n \right) \\
&\leq c P \left( n \in \tau^{T_N} \right) \leq \frac{C'}{\sqrt{n}}.
\end{aligned}
\end{equation}
We get that $e^{-n \phi(\delta_N,T_N)} \leq c$ thanks to \eqref{phi} and noting that $n\le 1/{\delta_N^2} \le T_N^2$. Thus, in this regime, our renewal measure behaves as in the simple random walk case.
\subsubsection{Upper bound for the intermediate range \label{Section 7.3.2}}
We are now looking to prove Proposition \ref{prop 1.1} :
\begin{proposition}
There exists $C>0$ such that, for all $N \in \mathbb{N}$ and $ 1/\delta_N^2 \leq n \leq T_N^3 \delta_N $:
\begin{equation}
\proba{n \in \tau} \leq \frac{C}{\delta_N^2 \min\{n^{3/2},T_N^3 \} }.
\label{1.0}
\end{equation}
\label{prop 1.1}
\end{proposition}
We distinguish between two cases :\\

\par (i) Let us consider $ 1/\delta_N^2 \leq n \leq 2T_N^2$. Because $n$ is smaller than $2T_N^2$ and by \eqref{phi}, it comes that $e^{-n \phi(\delta_N,T_N)} \leq e^c$. Remind \eqref{2.4}:

\begin{equation}
\proba{n \in \tau} \leq e^c \somme{k=0}{\infty}P(\tau_k^{T_N} = n) e^{-k \delta_N},
\end{equation}
dawith $\tau_k^{T_N}$ defined in \eqref{4.5000}. By applying the  inequality in Proposition \ref{proposition 4.1}, we can bound from above the latter quantity by $\frac{C}{n^{3/2} \delta_N^2}$, which is the result we are looking for.\\

\par (ii) Let us consider $ 2T_N^2 \leq n\leq T_N^3 \delta_N$. We are first going to bound from above the contribution coming from a certain type of trajectories, namely the ones with only "small" excursions, i.e. excursions smaller than $T_N^2$. Remind that $L_n = \inf \{ l \in \mathbb{N}: \tau_{l} \geq n \}$ is the number of excursions needed to exceed $n$. We use the notations introduced in \eqref{4.20}. We write $\xi_i := \tau_i^{T_N} - \tau_{i-1}^{T_N}$. Thus, it comes:
\begin{equation}
\proba{n \in \tau , \, \, \xi_i \leq T_N^2 \, \, \forall i \leq L_n} = \somme{k=0}{\infty}Q_k^{T_N}(n) e^{-k \delta_N} e^{-n \phi(\delta_N,T_N)}.
\end{equation}
We use the upper bound in \eqref{8.6}, that is valid for $n \geq 2T_N^2$, and the fact that $e^{-n(\phi+g(T_N))} \leq 1$ to write:
\begin{equation}
\proba{n \in \tau, \, \, \xi_i  \leq  T_N^2 \, \, \forall i \leq L_n} \leq \somme{k=0}{\infty} \frac{C}{T_N^3}\left( 1 + \frac{C'}{T}\right)^k e^{-k\delta_N} \leq \frac{C}{T_N^3 \delta_N},
\end{equation}
because $\delta_N \gg_N 1/T_N$. Now, we have to prove the equation in Proposition \ref{prop 1.1} for $n \in [2T_N^2,T_N^3 \delta_N]$, and for the trajectories with at least an excursion bigger than $T_N^2$. We are going to decompose the sum with the number of excursions needed to reach $n$, and the number of excursions bigger than $T_N^2$ among those excursions:
\begin{equation}
\begin{aligned}
&\proba{ \{ n \in \tau \} \cap \{ \exists i \leq L_n : \tau_i - \tau_{i-1}  \geq  T_N^2 \}} \\
&= \somme{k=1}{\infty} e^{-k \delta_N} e^{n \phi} \somme{j=1}{k} P(\tau_k^{T_N}=n, \# \{ l \leq k : \xi_l  \geq  T_N^2\} = j \}) \\
& \leq 
\somme{k=1}{\infty} e^{-k \delta_N} e^{n \phi} \somme{j=1}{k} \binom{k}{j} 
\somme{\substack{T_N^2 \leq l_2,...,l_j  \leq  n \\ l_2+...+l_j \leq n-T_N^2}}{} P(\xi_2=l_2) ... P(\xi_j=l_j)\\
&\hspace{2.5cm} \times\somme{l=1}{n-l_2-...-l_j-T_N^2}Q_{k-j}^{T_N}(l) P(\tau_1^{T_N} = n-l-l_2-...-l_j).
\end{aligned}
\label{5.5000}
\end{equation}
Here is a lemma to bound from above this sum:
\begin{lemme}
The following upper bounds hold for $2 T_N^2 \leq n \leq \delta_N T_N^3$ and $l_2,...,l_j \in \mathbb{N}$ such that $n-l_2-...-l_j \geq T_N^2$:
\begin{equation}
\begin{aligned}
&\somme{\substack{T_N^2 \leq l_2,...,l_j  \leq  n \\ l_2 + ... + l_j \leq n-T_N^2}}{} P(\xi_2=l_2) ... P(\xi_j=l_j) 
&\leq \left( \frac{c_2}{T_N^3} \right)^{j-1} \somme{\substack{ l_2 + ... + l_j \leq n-T_N^2}}{} e^{-(l_2+...+l_j)g(T_N)},
\label{5.6000}
\end{aligned}
\end{equation}
\begin{equation}
\begin{aligned}
 &\somme{l=1}{n-l_2-...-l_j-T_N^2}Q_{k-j}^{T_N}(l) P(\tau_1^{T_N} = n-l-l_2-...-l_j) \\
 &\leq C\left(1 + \frac{C''}{T_N} \right)^k  \frac{e^{-g(T_N)(n-l_2-...-l_k)}}{T_N^3}.
\end{aligned}
\label{5.7}
\end{equation}
\label{lemme 5.1}
\end{lemme}
The proof is technical, and is deferred to Appendix \ref{Appendix E}. By using these two upper bounds, the fact that $n \leq T_N^3 \delta_N$, which implies that $-n(\phi + g(T_N)) \leq c$, and finally that ${\delta_N} \gg_N {1}/{T_N}$, we obtain:
\begin{equation}
\begin{aligned}
\text{\eqref{5.5000}} &\leq \somme{k=1}{\infty}e^{-k\delta_N} C \left( 1 + \frac{C''}{T_N}\right)^k \frac{1}{T_N^3} e^{n(\phi - g(T_N))} \somme{j=1}{k} \binom{k}{j}\left( \frac{c}{T_N^3} \right)^{j-1} \somme{l_2+...+l_j\leq n}{}1 \\
&\leq \frac{C}{T_N^3}\somme{k=1}{\infty}e^{-k\delta_N/2}\somme{j=1}{k}\binom{k}{j} \left(\frac{c}{T_N^3}\right)^{j-1}\binom{n}{j-1} \\
&\leq  \frac{C}{T_N^3}\somme{k=1}{\infty}e^{-k\delta_N/2}\somme{j=1}{k}\binom{k}{j} \left(\frac{cn}{T_N^3}\right)^{j-1}\frac{1}{(j-1)!}.
\end{aligned}
\end{equation}
The computation of the sum is standard. Let us remind what the result is:
\begin{lemme}
For $0<\alpha<\lambda<1$:
\begin{equation}
\somme{k=1}{\infty}\lambda^k \somme{j=1}{k}\binom{k}{j}\alpha^{j-1}\frac{1}{(j-1)!} = \frac{\lambda}{(1-\lambda)^2}e^{\frac{\alpha \lambda}{1-\lambda}}.
\end{equation}
\label{lemme 5.2}
\end{lemme}
\noindent The proof is left to the reader. Thanks to this result, with $\lambda = e^{-\delta_N/2}$ and $\alpha = \frac{cn}{T_N^3}$, and using that $(1-\lambda)$ is equivalent to $ {\delta_N}/{2}$:
\begin{equation}
\text{\eqref{5.5000}} \leq \frac{C}{\delta_N^2 T_N^3} e^{\frac{n}{\delta_N T_N^3}}.
\end{equation}
Hence, for $n \leq T_N^3\delta_N$, we have proven the equation \eqref{1.0}.
\section{ Proof of the second theorem: the remaining regimes \label{Section 123.77}}
We first state the needed technical lemmas in Section \ref{Section 123.128}. Then, Sections \ref{123.999} to \ref{Section 123.7.4} contain the proof of Theorem \ref{Theoreme 2} and \ref{Pomme de terre Th critique}.
\subsection{ Technical results \label{Section 123.128}}
Recall \eqref{g}. We have the following asymptotic bounds:
\begin{lemme}
    There exists $T_0$ in $\mathbb{N}$, $c_7,c_8 > 0$ depending on $\beta$, $a$ and $b$ such that, when $b \geq a$ and for all $T_N  \geq  T_0$:
\begin{equation}
    \proba{m \leq \tau_1 \leq n} \geq \frac{c_7}{T_N}\left( e^{-(g(T_N) + \phi(\delta_N,T_N))m} -  e^{-(g(T_N) + \phi(\delta_N,T_N))n}\right)
\label{minoration de la proba que tau 1 soit entre m et n, quand b geq a}
\end{equation}
\begin{equation}
    \proba{\tau_1 \geq m} \leq \frac{c_8}{\min\{ T_N, \sqrt{m} \}} e^{-(g(T_N) + \phi(\delta_N,T_N))m}.
\label{majoration de la proba que tau 1 soit plus grand que m, cas b geq a }
\end{equation}
\end{lemme}
\noindent Moreover, we have, with $x_\beta$ defined in  \eqref{def x_beta}:
\begin{equation}
g(T_N) + \phi(\delta_N,T_N) =
\begin{cases}
  \frac{\pi^2}{2 T_N^2}(1 + o_N(1)) & \text{if } a < b, \\
  \\
  \frac{\pi^2 - x_\beta^2}{2 T_N^2}(1 + o_N(1)) & \text{if } a = b.
\end{cases}
\label{g + phi, cas b geq a}
\end{equation}
\begin{proof} The proof of \eqref{minoration de la proba que tau 1 soit entre m et n, quand b geq a} and \eqref{majoration de la proba que tau 1 soit plus grand que m, cas b geq a } is done by summing  \eqref{probabilité que tau 1 = n}, using  \eqref{3.3000} and \eqref{3.4000} or  \eqref{g}. Equation \eqref{g + phi, cas b geq a} is obtained by combining \eqref{3.3000},  \eqref{3.4000} and  \eqref{g}.
\end{proof}
The next lemma is about the first and second moments of $\tau_1$ under $\probaRenouvellementSansParenthese$ and the probability to switch interface. It is proven in Appendix \ref{Annex E.22}:
\begin{lemme}
With $x_\beta$ defined in  \eqref{def x_beta}:
\begin{equation}
\EsperanceRenouvellement{\tau_1} =
\begin{cases}
   T_N(1 + o_N(1)) & \text{if } b > a, \\
   \frac{T_N\beta}{x_\beta^2}
   \left(
     1 + \frac{x_\beta}{\sin x_\beta}
     + o_N(1)
   \right) & \text{if } b = a.
\end{cases}
\label{moment d'ordre 1 de tau 1, b geq a}
\end{equation}
\begin{equation}
\EsperanceRenouvellement{\tau_1^2} =
\begin{cases}
  \frac{T_N^3}{3} + o_N(N^{3a}) & \text{if } b > a, \\
  \\
  \frac{T_N^3 \beta}{x_\beta^3}
  \left( 
    \frac{\beta}{\sin (x_\beta)} + \frac{1}{\sin (x_\beta)} - \frac{1}{x_\beta}
  \right)(1 + o_N(1)) & \text{if } b = a.
\end{cases}
\label{moment d'ordre 2 de tau 1, b geq a}
\end{equation}
\begin{equation}
\label{Probabilité de changer d'interface,b geq a}
\proba{\varepsilon_1^2 = 1} =
\begin{cases}
  \frac{1}{T_N}(1 + o_N(1)) & \text{if } b > a, \\
  \\
  \frac{x_\beta}{T_N \sin x_\beta}(1 + o_N(1)) & \text{if } b = a.
\end{cases}
\end{equation}
\label{lemme 123.5.654}
\end{lemme}
\noindent We observe that, when $b>a$, the renewal process behaves like the simple random walk at first order. When $b=a$, the asymptotics are the same up to a multiplicative constant. 
\begin{proposition} There exists $T_0$ in $\mathbb{N}$, $c_9,c_{10} > 0$ depending on $\beta$, $a$ and $b$ such that, when $b \geq a$ and for all $T_N  \geq  T_0$ and $n \in 2\mathbb{N}$:
\begin{equation}
\frac{c_9}{\min \{\sqrt{n}, T_N \} } 
\leq \proba{n \in \tau} \leq 
\frac{c_{10}}{\min \{\sqrt{n}, T_N \} }
\label{Proposition 123}.
\end{equation}
\label{proposition 123}
\end{proposition}
\noindent Note that, for the simple random walk, $P(n \in \tau^T) \asymp_n  1/{\min \{\sqrt{n},T \}}$, as proven in Lemma \ref{lemme 8.1}. The proof of Proposition \ref{proposition 123} is very similar to what has already been done for Proposition \ref{Prooposition 1}, so we write it up more briefly.
\begin{proof}[Proof of Proposition \ref{proposition 123}] \phantom{pourquoi pas manger des pommes de terre ce soir ?}

\begin{enumerate}
    \item \textbf{Regime $n \leq T_N^2$.} The exact same computations as in Section \ref{Section 8.11} give the upper bound for $n \leq T_N^2$. For the lower bound, we use the exact same ideas as in Section \ref{section 7}. The only change is $s_n$, defined after \eqref{7.3}, that becomes $n$, and the denominator in the rightmost term of inequality \eqref{7.111} becomes ${\sqrt{n}}$. 
    \item \textbf{Regime $n  \geq  T_N^2$.} We now know that Proposition \ref{proposition 123} is true for $n \leq T_N^2$. Our aim is to extend this proposition to $2\mathbb{N}$. To do so, we do the same computations as in Section \ref{section 6}. Let us list the few main changes: 
    \begin{enumerate}
    \item \textbf{Lower bound}
        \begin{enumerate}
            \item  Replace $T_N^3 \delta_N^2$ by $T_N$ in \eqref{Bourgogne 1}.
            \item Replace $T_N^3 \delta_N$ by $T_N^2$ from \eqref{Bourgogne 1} to \eqref{Bourgogne 6}.
            \item Use Proposition \ref{proposition 123} instead of Proposition \ref{Prooposition 1} below \eqref{Bourgogne 1}. 
        \end{enumerate}
    \item \textbf{Upper bound}
    \begin{enumerate}
        {%
        \item Consider the interval $[n-T_N^2, n-T_N^2/2]$ instead of $[n-T_N^3\delta_N, n-T_N^2]$ from \eqref{7.1} to \eqref{Bourgogne 16} as well as in \eqref{Bourgogne 17}.
        \item Use Proposition \ref{proposition 123} instead of Proposition \ref{Prooposition 1} below \eqref{7.1}, \eqref{Bourgogne 13} ad \eqref{Bourgogne 18}. 
        \item Replace $T_N^3 \delta_N^2$ by $T_N$ in \eqref{7.1} and \eqref{Bourgogne 11}.
        \item Replace $c \delta_N$ by ${c}/{T_N}$ in \eqref{Bourgogne 15}.
        \item Replace ${c}/{(T_N^3 \delta_N)}$ by ${c}/{T_N^2}$ in \eqref{Bourgogne 16}. 
        \item Replace ${c}/{\delta_N}$ by $c T_N$ in \eqref{Bourgogne 30}.
        \item Other minor changes are left to the reader.
        }%
    \end{enumerate}
\end{enumerate}
\end{enumerate}
\end{proof}
\subsection{ Diffusive regime under close interfaces \label{123.999} }
We now prove Parts \ref{partie 123.1 du théorème} and \ref{partie 123.2 du théorème} in Theorem~\ref{Theoreme 2} and \ref{cas a=b<1/2} in Theorem \ref{Pomme de terre Th critique} by closely following the proof of~\cite[Theorem 1.1, Part 1]{Caravenna2009depinning} in Section 3 therein. We only list the changes and advise the reader to check the four steps below, first in the case $a \geq b$ then in the case $a<b$. The quantities $v_\delta$ and $k_N$ in \cite[Section 3]{Caravenna2009depinning} respectively become $v_\delta = \proba{\epsilon_1^2 = 1}$ and $ k_N := N/\EsperanceRenouvellement{\tau_1}$. We also define  $s_N := \max \{ T_N^2, T_N^3 \delta_N \}$. 

\begin{enumerate}
    \item \textbf{\cite[Section 3.1]{Caravenna2009depinning}} We repeat this step, by using the formulas established in Section \ref{section 5000} (case $a > b$) and Section \ref{Section 123.128} (case $a \leq b$) to apply the Berry-Esseen theorem.
    \item \textbf{\cite[Section 3.2]{Caravenna2009depinning}} First, we assume that $V_N$ defined in the first line of \cite[Section 3.2]{Caravenna2009depinning} satisfies $s_N \ll V_N \ll N$. Sticking with the notation of \cite{Caravenna2009depinning}, we see  that the part dealing with $A_\nu^N$, see \cite[Eq. (3.8)]{Caravenna2009depinning}, does not present any additional difficulty. To prove \cite[Eq (3.9)]{Caravenna2009depinning}, the reader is invited to check that $\EsperanceRenouvellement{(Y_{k_N + \nu k_N} - Y_{k_N})^2} = \nu \,  v_\delta \, k_N$.
    \item \textbf{\cite[Section 3.3]{Caravenna2009depinning}}  We point out two changes here. First, $L_{N-p-T_N^3}$ must be simply replaced by $L_{N-p-s_N}$ in \cite[Eq. (3.11)]{Caravenna2009depinning}. However, proving that \cite[Eq. (3.10)]{Caravenna2009depinning} is equivalent to \cite[Eq. (3.11)]{Caravenna2009depinning} is harder in our case.  We thus display this part of the proof below. Then, going from (3.16) to (3.21) in \cite{Caravenna2009depinning} does not present any  additional difficulty.
    \item \textbf{\cite[Section 3.4]{Caravenna2009depinning}} There is no noticeable change here, despite a few laborious computations.
\end{enumerate}
Let us now prove the missing part in the third item above, that is going from \cite[Eq (3.11)]{Caravenna2009depinning} to \cite[Eq (3.16)]{Caravenna2009depinning}. We assume that the reader has already checked Sections 3.1 and 3.2 in \cite{Caravenna2009depinning}, and the first four lines of \cite[Section 3.3]{Caravenna2009depinning} and \cite[Eq (3.10)]{Caravenna2009depinning}. We adopt the notations used between \cite[Eq (3.11)]{Caravenna2009depinning} to \cite[Eq (3.16)]{Caravenna2009depinning} and continue the proof from \cite[Eq (3.10)]{Caravenna2009depinning}. A first observation is that we can safely replace $L_{N-\po}$ with $L_{N-\po-s_N}$. To prove this, since $k_{N} \rightarrow \infty$, the following bound is sufficient: there exists $c>0$ and ${\rm cst} > 0$ such that, for every $N, M \in 2 \mathbb{N}$
\begin{equation}
\sup _{\po \in\left\{0, \ldots, V_{N}\right\} \cap 2 \mathbb{N}} \proba{\left|Y_{L_{N-\po}}-Y_{L_{N-\po-s_N }}\right| \geq M \Big| N-\po \in \tau} \leq { \rm cst} (1-c)^M.
\label{Bourgogne 31}
\end{equation}
We define $\Tilde{\tau} := \{ \tau_i \in \tau : \varepsilon_i^2 = 1 \}$, 
and we write $\Tilde{\tau}_1, \Tilde{\tau}_2,...$ its elements. One can check that 
the left-hand side of \eqref{Bourgogne 31} is bounded from above by the quantity 
$\proba{\#\Tilde{\tau} \cap\left[N-\po-s_N, N-\po \right) \geq M \mid N-\po \in 
\tau}$. By using time-inversion and the renewal property, we then rewrite this as
\begin{equation}
\begin{aligned}
& \proba{\#\left\{\tilde{\tau} \cap\left(0, s_N\right]\right\} \geq M \mid N-\po \in \tau}
=\proba{\Tilde{\tau}_M \leq s_N \mid N-\po \in \tau} \\
& \leq \sum_{n=1}^{s_N} \proba{\Tilde{\tau}_M = n} \cdot \frac{\proba{N-\po-n \in \tau}}{\proba{N-\po \in \tau}}.
\label{123.2.13}
\end{aligned}
\end{equation}
Recalling that $N \gg V_{N} \gg s_N$ and using the estimate in \eqref{Proposition 1}, we see that the ratio in 
the r.h.s. of \eqref{123.2.13} is bounded from above by a constant, uniformly in $0 \leq n \leq s_N$ and 
$\po \in\left\{0, \ldots, V_{N}\right\} \cap 2 \mathbb{N}$. We now have to estimate 
$\proba{\Tilde{\tau}_{M} \leq s_N}$. By writing $\xi_i = \tilde{\tau}_i - \Tilde{\tau}_{i-1}$, this probability is lower than $\proba{\forall i \leq M, \xi_i \leq s_N}$. We first note that $\xi_i$ stochastically dominates the variable $\tau_1$ conditioned on $\varepsilon_1^2=1$, by removing all 
the excursions that stay on the same interface. Therefore, by denoting $(X_i)_{i \in \mathbb{N}}$ a sequence of i.i.d. random variables distributed as $ \tau_1 $ conditioned on $ \varepsilon_1^2=1$, we have that
\begin{equation}
\proba{\forall i \leq M, \xi_i \leq s_N} \leq \proba{X_1 \leq s_N}^M. 
\end{equation}
We now need a lemma to control $\proba{\tau_1 = n, \varepsilon_1^2 = 1}$.
\begin{lemme}
    Recall the constants $c_1,c_2$ from \eqref{probabilité pour la marche aléatoire simple que tau 1 vaille n}. There exist $c_{11},c_{12} > 0$ such that, for all $T \in 2\mathbb{N}$ and $n \geq T^2$:
    \begin{equation}
    P \left(\tau_1^T = n, \left(\varepsilon_1^T\right)^2 = 1 \right) \geq \frac{c_1}{2 T^3} e^{-ng(T)} - \frac{c_2}{T^3} e^{-4ng(T)(1+O_T(T^{-2}))},
\label{123.2.14}
    \end{equation}
    \begin{equation}
    \begin{aligned}
    \proba{\tau_1 = n, \varepsilon_1^2 = 1} &\geq  \frac{c_{11}}{T_N^3} e^{-n(g(T_N)+\phi(\delta_N,T_N))} \\&- \frac{c_{12}}{T_N^3} e^{-4ng(T_N)(1+O_{T_N}(T_N^{-2})) -n \phi(\delta_N,T_N)}.
    \label{123.4}
    \end{aligned}
    \end{equation}
    \label{lemme 123.222}
\end{lemme}
\begin{proof}
We consider $T \in 4\mathbb{N}$, but the proof can be adapted without extra difficulty to the case $T \in 2\mathbb{N}$. Let us define 
\begin{equation}
    \begin{aligned}
A &:= \{\forall \, 1 \leq i \leq n-1, S_i \notin \{ -T,0,T \} , S_n = 0 \},\\
B &:= \{\forall \, 1 \leq i \leq n-1, S_i \notin \{ -T/2,0,T/2 \} , S_n = 0 \},\\
C &:= \{\forall \, 1 \leq i \leq n-1, S_i \notin \{ -T,0,T \} , |S_n| = T \}. 
    \end{aligned}
\end{equation}
By the reflection principle applied at the first time  that $S_n$ reaches $\pm T/2$, one can see that there are as many trajectories touching their first interface at $\pm T$ at time $n$ as trajectories touching their first interface at $0$ and touching $\pm T/2$ before. Hence, $P(C) = P(A \backslash B) = P(A) - P(B)$. Remind that $q_T^i(n) = P(\tau_1^T = n, \left(\varepsilon_1^T \right)^2 = i)$. We deduce that $2 q_T^1(n) = q_T^0(n) - q_{T/2}^0(n)$. Hence, $q_T^0(n) \geq \frac{1}{2}P(\tau_1^T=n)$. Moreover, $q_{T/2}^0(n) \leq P(\tau_1^{T/2}=n)$, so $2q_T^1(n) \geq \frac{1}{2}q_T(n) - P(\tau_1^{T/2}=n)$. Let us now use  \eqref{probabilité pour la marche aléatoire simple que tau 1 vaille n} and $\eqref{g}$. We obtain: 

\begin{equation}
       P(\tau_1^T = n, \varepsilon_1^2 = 1) \geq \frac{c_1}{2 T^3} e^{-ng(T)} - \frac{c_2}{T^3} e^{-4ng(T)(1+O_T(T^{-2}))}. 
\end{equation}
Equation \eqref{123.4} directly follows from \eqref{2.4} and \eqref{123.2.14}. 
\end{proof}
\begin{lemme}
When $a>b$, there exists $c_{13}>0$ and $C>1$ such that, for all $m \geq C$:
\begin{equation}
\proba{\tau_1 \geq m s_N | \varepsilon_1^2 = 1} \geq c_{13} e^{-m}.
\label{123.2.16}
\end{equation}
\label{Lemme Bourguignon}
\end{lemme}
\begin{proof}
Combining  \eqref{g + phi} and \eqref{probabilité de changer d'interface} ($a>b$) or \eqref{g + phi, cas b geq a} and \eqref{Probabilité de changer d'interface,b geq a} ($a \leq b$), and \eqref{123.2.14}:
\begin{equation}
\begin{aligned}
&\proba{\tau_1 \geq m s_N | \varepsilon_1^2 = 1} 
= 
\frac{1}{\proba{\varepsilon_1^2 = 1}}\somme{ k\ge m s_N}{} \proba{\tau_1 = k, \varepsilon_1^2 = 1 } \\
&\geq c \min  \Big\{ T_N, \frac{1}{\delta_N} \Big\} \somme{ k\ge m s_N}{} \frac{1}{T_N^3} \left(e^{-\frac{k}{s_N}} - c' e^{-\frac{3k}{T_N^2}}  \right)
\geq c_{13} e^{-m}.
\end{aligned}
\label{123.1.6}
\end{equation}
We used in the last equation that $e^{-m} - c' e^{-\frac{3ms_N}{T_N^2}} \geq (1/2) e^{-m}$, which is the case for $m$  larger than a certain real number $C \geq 1$ independent of $a$ and $b$.
\end{proof}
Coming back to the number of interface changes before $s_N$, with $C$ defined in Lemma \ref{Lemme Bourguignon}, and using \eqref{123.2.16}:
\begin{equation}
\begin{aligned}
\proba{X_1 \leq C s_N} &= 1-P(X_1 \geq C s_N) =1- \proba{\tau_1 \geq C s_N | \varepsilon_1^2 = 1} 
&\leq 1 - c. 
\end{aligned}
\label{123.1.7}
\end{equation}
Thus, putting everything together, $\proba{\tilde{\tau}_{M} \leq s_N} \leq (1-c)^M$.
\subsection{Simple random walk regime under far-away interfaces}
In this section, we deal with the case $b \geq 1/2$  and $a\geq 1/2$. When $b=1/2$, the repulsion still has a (weak) effect, but when $b>1/2$, the repulsion does not matter anymore.
\subsubsection{ Diffusive regime under far-away interfaces  \label{Section 123.7.1} }
Let us first prove that, when $b = 1/2$ and $a > 1/2$, the polymer does not touch any interface other than  the one at the origin.
\begin{lemme}  One has:
\begin{equation} 
\limite{N}{\infty} \probaPolymere{\exists i \leq N : S_i \in T_N\mathbb{Z} \backslash \{0\}} =
0. 
\end{equation}
\label{Lemme 123.8.2}
\end{lemme}
\begin{proof}
First, using \eqref{Markov exponentiel} and \cite[Theorem 11, Chapter III]{Petrov} \footnote{ In that theorem, the superior sign should be an inferior sign.} for the simple random walk:
\begin{equation}
P\left(\underset{i \leq N}{\max} |S_i| \geq T_N \right) \leq 4 \exp \left( \frac{T_N^2}{8N} \right) = o_N(1).
\label{123.8.6}
\end{equation}
Now, let us focus on the polymer measure. Using that the partition function is bounded from below by a constant thanks to \eqref{123.1.88}, and recalling  \eqref{1.2}, we then have:
\begin{equation}
\begin{aligned}
\probaPolymere{\exists i \leq N,  S_i = \pm T_N  }&\leq 
\frac{P(\exists i \leq N,  S_i = \pm T_N)}{Z_{N,\delta_N}^{T_N}}
&\asymp_N P(\exists i \leq N,  S_i = \pm T_N)  \\
&= o_N(1).
\end{aligned}
\end{equation}
\end{proof}
Let us come back to the simple random walk. Let $u  \leq  v \in \overline{\R}$ and let us denote $K_{u,v}^N := P( u \leq {S_N}/{\sqrt{N}} \leq v)$. By decomposing the trajectory of the random walk according to the last contact with an interface before $N$, and dismissing the trajectories touching any interface other than the origin thanks to \eqref{123.8.6}:
\begin{equation} 
\begin{aligned}
K_{u,v}^N &= \somme{k=1}{N-1}
P(S_k = 0) P \left( u \leq \frac{S_{N-k}}{\sqrt{N}} \leq v, S_i \notin \{-T_N,0,T_N\}\  \forall i \leq N-k \right) \\
&+ o_N(1).
\end{aligned}
\label{123.8.7}
\end{equation}
Finally, we use \eqref{3.13}:
\begin{equation}
\begin{aligned}
&\probaPolymere{u \leq \frac{S_N}{\sqrt{N}} \leq v} 
\\
&= o_N(1)  + \somme{k=1}{N-1}  \frac{ \proba{ \exists i \geq 1 \colon k = \tau_i, \, \varepsilon_1+...+\varepsilon_i = 0} e^{-k \phi(\delta_N,T_N)}}{Z_{N, \delta_N}^{T_{N}}} \\
&\qquad \qquad    \qquad  \times P \Big( u \leq \frac{S_{N-k}}{\sqrt{N}} \leq v, S_i \notin \{0, \pm T_N\} \forall i \leq N-k \Big). 
\end{aligned}
\label{123.8.9}
\end{equation}
Note that $e^ {-k \phi(\delta_N,T_N)} \asymp_N 1 $ because of \eqref{3.3000}, \eqref{3.4000} or \eqref{phi}, and the same holds for $Z_{N, \delta_N}^{T_{N}}$ thanks to \eqref{123.1.88}. Therefore, we can dismiss them. Now, we want to show that :
\begin{lemme}
Uniformly in $k\le N$,
\begin{equation}
\proba{ \exists i  \ge 1 \colon k = \tau_i, \, \varepsilon_1+...+\varepsilon_i = 0}  \asymp_N \proba{k \in \tau}.
\end{equation}
\end{lemme}
\begin{proof} First, $\proba{ \exists i, k = \tau_i \text{, } \varepsilon_1+...+\varepsilon_i = 0} \leq \proba{k \in \tau}$ because of the first event being included in the other one. Recall Definition \ref{definition 1}. Then, we use that:
\begin{equation}
\begin{aligned}
&\proba{ \exists i \geq 1 \colon k = \tau_i, \varepsilon_1+...+\varepsilon_i = 0} \\
&\geq \proba{k \in \tau} - \proba{\exists i\leq L_N: \varepsilon_i \neq 0}.
\label{Bourgogne 21}
\end{aligned}
\end{equation}
Equation \eqref{3.12} indicates that we can compare the  polymer measure with the renewal measure. We combine it with a rough use of \eqref{Proposition 1}, showing that for $r \leq N$, $\proba{r \in \tau} \geq \frac{1}{N}$. Hence:
\begin{equation}
\begin{aligned}
\proba{\exists i\leq L_N:\varepsilon_i \neq 0} &= \somme{k=1}{N} \proba{\exists i \leq L_N \colon \varepsilon_i \neq 0, k \in \tau } \proba{\tau_1 \geq N-k} \\
&\leq \somme{k=1}{N} \proba{\exists i \leq L_N \colon \varepsilon_i \neq 0 | k \in \tau } \frac{1}{\proba{k \in \tau}} \\
&\leq N \somme{k=1}{N} \probaPolymere{\exists i \leq N : S_i \in T_N \mathbb{Z} \backslash \{0\} | k \in \tau }  \\
&\leq 
N^3 \probaPolymere{\exists i \leq N : S_i \in T_N \mathbb{Z} \backslash \{0\}}.
\end{aligned}
\end{equation}
Now, combining \eqref{123.8.6} and \eqref{123.1.88} gives that $\probaPolymere{\exists i \leq N : S_i \in T_N \mathbb{Z}} \leq e^{-cN}$ for a certain $c>0$. Hence, coming back to \eqref{Bourgogne 21}:
\begin{equation}
\begin{aligned}
\proba{ \exists i \colon k = \tau_i \text{, } \varepsilon_1+...+\varepsilon_i = 0} 
& \geq \proba{k \in \tau}  - \proba{\exists i\leq N: S_i \in T_N \mathbb{Z}} \\
&\asymp_N \frac{1}{\sqrt{k}} - N^2e^{-cN} \asymp_N \frac{1}{\sqrt{k}} \asymp_N \proba{k \in \tau}.
\label{Pomme 6.26}
\end{aligned}
\end{equation}
\end{proof}
\noindent Now, combining \eqref{Proposition 1} and standard facts about the simple random walk gives that $\proba{k \in \tau} \asymp_N \frac{1}{\sqrt{k}}  \asymp_k P(S_k = 0)$. Hence, \eqref{123.8.9} becomes:
\begin{equation}
\begin{aligned}
\probaPolymere{ u \leq \frac{S_N}{\sqrt{N}} \leq v } &\asymp_N \somme{k=1}{N-1} P( S_k = 0) P \left( u \leq \frac{S_{N-k}}{\sqrt{N}} \leq v, S_i \neq 0 \,\ \forall i \leq N-k \right) \\
&+ o_N(1). 
\end{aligned}
\label{123.8.11}
\end{equation}
Comparing \eqref{123.8.11} and \eqref{123.8.7} gives us that $\probaPolymere { u \leq \frac{S_N}{\sqrt{N}} \leq v } \asymp_N K_{u,v}^N$. The proof is therefore over.

\subsubsection{ Simple random walk regime under very weak repulsion \label{Section 123.7.2}}

We now  assume that $b > 1/2$ and $a \geq 1/2$. As we will see, the repulsion is now too weak to have an effect. 

\par Recalling Proposition \ref{Proposition 123.1.3} and and \eqref{3.3000}, \eqref{3.4000} and \eqref{phi}, $Z_{N,\delta_N}^{T_N} \longrightarrow 1$. Recall Definition \ref{definition 1}. Let us prove a lemma about the number of contacts between the simple random walk and interfaces:
\begin{lemme}
 Assume $a \geq 1/2$. With the constant $c_4$ as in \eqref{probabilité pour la marche aléatoire simple que tau 1 soit plus grande que n} and $c' = c_4^{-1} e^{N g(T_N)}$,
\begin{equation}
P\Big(L_N \geq c' \sqrt{N} \ln(N) \Big) \leq \frac{1}{N}.
\end{equation}
\end{lemme}
\begin{proof} Using  \eqref{probabilité pour la marche aléatoire simple que tau 1 soit plus grande que n}, we have $P\Big(\tau_1^{T_N} \geq T_N^2 \Big) \geq \frac{1}{c' \sqrt{N}}$ because $T_N \geq \sqrt{N}$. Hence:
\begin{equation}
\begin{aligned}
P\Big(L_N \geq c' \sqrt{N} \ln(N) \Big) \leq 
P\Big(\tau_1^{T_N} \leq N\Big)^{c' \sqrt{N} \ln(N)} 
&\leq \Big(1 - \frac{1}{c' \sqrt{N}} \Big)^{c' \sqrt{N} \ln(N)} \leq \frac{1}{N}.
\end{aligned}
\end{equation}
\end{proof}
Hence, for any event $A$, $P(A \cap L_N \leq \sqrt{N} \ln(N)/c') \geq P(A) + o_N(1)$. Thus:
\begin{equation}
\begin{aligned}
\probaPolymere{A} = \frac{o_N(1) + E\left(1\left(A  \cap L_N \leq \sqrt{N} \ln(N)/c'\right) e^{-\delta_N L_N} \right)}{Z_{N,\delta_N}^{T_N} } =
P(A) + o_N(1).
\end{aligned}
\end{equation}
having used that $ e^{-\delta_N L_N} \sim_N 1$ because $L_N \leq \sqrt{N} \ln(N)/c'$ and $b>\frac{1}{2}$.
\subsection{ Third border case \label{Section 123.7.4}}
We now prove part \ref{poisson de pomme de terre} of Theorem \ref{Pomme de terre Th critique}. We henceforth assume that $3a-b = 1$ and $a<1/2$. 
Let us start by a lemma estimating $\EsperanceRenouvellement{\tau_1|\varepsilon_1 = 1}$.
\begin{lemme}
    When $a>b$:
\begin{equation}
 \EsperanceRenouvellement{\tau_1 \mathbf{1}\{\varepsilon_1^2 = 1\} } \sim_N \frac{T_N^3 \delta_N^2}{4 \pi^2},
 \label{123.8.29}
\end{equation}
\begin{equation}
\EsperanceRenouvellement{\tau_1|\varepsilon_1^2 = 1} \sim_N 
\frac{T_N^3 \delta_N}{2 \pi^2}.
\label{123.8.30}
\end{equation}

\end{lemme}
\begin{proof} Recall the definition of $\gamma$ in \eqref{definition de gamma}.  Using the equation above (A.5) in \cite{Caravenna_2009} and following (A.7) and (A.9) in \cite{Caravenna_2009}:
\begin{equation}
 \EsperanceRenouvellement{\tau_1 \mathbf{1}\{\varepsilon_1^2 = 1\}} =
 - 2e^{\delta_N} \left(\Tilde{Q}_{T_N}^{1}\right)' (\gamma(\phi(\delta_N,T_N))) \gamma'(\phi(\delta_N,T_N)),
 \label{123.8.31}
\end{equation}
with
\begin{equation}
\Tilde{Q}_{T_N}^1(\gamma) = \frac{\tan \gamma}{2 \sin(T_N \gamma)}.
\end{equation}
Hence, with the computation of $\gamma'$ done below (A.12) in \cite{Caravenna_2009} and recalling \eqref{phi}:
\begin{equation}
\gamma'(\phi(\delta_N,T_N)) \sim_N \frac{T_N}{\pi}. 
\label{123.8.32}
\end{equation}
Moreover, we get from \eqref{A.16} that $\gamma(\phi(\delta_N,T_N)) = \frac{\pi}{T_N} - \frac{2\pi}{\delta_N T_N^2}(1 + o_N(1))$. Hence: 
\begin{equation}
Q_{T_N}^{1'}(\gamma) = \frac{1}{2 \cos^2(\gamma) \sin(T_N \gamma) } - \frac{T_N \tan(\gamma) \cos(T_N \gamma)}{2 \sin^2(T_N \gamma)} \sim \frac{\delta_N^2 T_N^2 }{8 \pi}.
\label{123.8.33}
\end{equation}
Putting \eqref{123.8.33}, \eqref{123.8.32} with \eqref{123.8.31}, we get \eqref{123.8.29}. Using \eqref{probabilité de changer d'interface}, we get \eqref{123.8.30}.
\end{proof}
\par Let us now prove the right side of \eqref{123.1.6p}.
\begin{proof} Remind that $T_N^3 \delta_N = \beta N$. In our case, Equation \eqref{123.8.30} gives us that $\EsperanceRenouvellement{\tau_1|\varepsilon_1^2 = 1} \sim_N 
\frac{\beta N}{2 \pi^2}$. Using \eqref{123.2.16} with $m=\frac{1}{\beta}$, one has that the probability of having more than $m$ interfaces change is smaller than $(1-c)^m$ for a certain constant independent of $N$. 

\end{proof}
Let us prove the left side of \eqref{123.1.6p}. We denote $\D{N}{m} := \{ \# \{i\le L_N\colon  \varepsilon_i^2 = 1\} = m \}$. By abuse of notation, we will use it for the renewal and for the polymer measure. We start by a technical lemma that we prove in Appendix \ref{Annexe C.45}:
\begin{lemme}
There exists $c>0$ such that, for all $m > 0$, there exists $N_0 > 0$ such that, for all $N  \geq  N_0$ and $k  \geq  {N}/{2}$:
\begin{equation}
\proba{\D{k}{m}} \geq (c/m)^m.  
\label{7.40}
\end{equation}
\label{Lemma 7.9}
\end{lemme}
\begin{proof} Let $m \in \N$. Remind \eqref{1.2}, \eqref{3.12} and \eqref{3.13}. We decompose the trajectory according to its last contact with an interface:
\begin{equation}
\begin{aligned}
\probaPolymere{ \D{N}{m}  } &\geq
\frac{1}{2}
\somme{k=N/2}{N} \frac{E\Big( e^{-H_{k,\delta_N}^{T_N}(S)} \mathbf{1}\{  \D{k}{m}\} \mathbf{1}\{ k \in \tau^{T_N} \}\Big)}{Z_{N,\delta_N}^{T_N}} P(\tau_1 \geq N-k) \\
&= \frac{1}{2}
\somme{k=N/2}{N} \frac{\mathbf{P}_{k,\delta_N}^{T_N} \left(\D{k}{m}  \big| k \in \tau^{T_N} \right) (Z_{k,\delta_N}^{T_N})^2 }{ E \Big( e^{-H_{k,\delta_N}^{T_N}(S)} \mathbf{1}\{k \in \tau^{T_N} \} \Big)  Z_{N,\delta_N}^{T_N}} P(\tau_1 \geq N-k) \\
& = 
\frac{1}{ 2}
\somme{k=N/2}{N} \frac{\proba{ \D{k}{m}  } (Z_{k,\delta_N}^{T_N})^2}{  e^{k \phi(\delta_N,T_N)}  Z_{N,\delta_N}^{T_N}} P(\tau_1 \geq N-k)
\end{aligned}
\end{equation}
First, $P(\tau_1 \geq N-k) \geq c/{T_N}$ by \eqref{probabilité pour la marche aléatoire simple que tau 1 soit plus grande que n}. Then, $ Z_{N,\delta_N}^{T_N} \asymp_N \frac{e^{N \phi(\delta_N,T_N)}}{\delta_N T_N}$ by \eqref{123.1.77}. Using \eqref{phi} and \eqref{7.40}, it comes:
\begin{equation}
\begin{aligned}
\probaPolymere{ S_N \geq m T_N  } &\geq
\frac{\rm cst}{\delta_N T_N^2} \somme{k=N/2}{N} \proba{ \D{k}{m} } e^{(k-N) \frac{1}{T_N^3 \delta_N} (1+o_N(1)) } \\
&\asymp_N T_N \Big(\frac{c}{m}\Big)^m.
\end{aligned}
\end{equation}
\end{proof}

\noindent Let us now prove \eqref{123.11.7}.

\begin{proof}

First, recall the notation introduced in \eqref{44.2}. Using \eqref{9.8} with $\cst = 0$ and $\C$ large, one has that the probability under the polymer measure that the last contact is done before ${\C}/{\delta_N^2}$ is greater than:
\begin{equation}
\probaPolymere {\BnuM_{0,\C}} \asymp_N \frac{1}{T_N \delta_N Z_{N,\delta_N}^{T_N}e^{-N\phi}} \Big( 1 - \frac{1}{\sqrt{\C}} \Big).
\label{123.6.32}
\end{equation}
Now, using the same idea as in  \eqref{44.4}, remembering that $L_N$ is the time of the last contact of the polymer with an interface (cf. Definition \ref{definition 1}), and denoting $\phi := \phi(\delta_N,T_N)$:
\begin{equation}
\begin{aligned}
&\probaPolymere{\tau_{L_N}^{T_N} \geq (1-\varepsilon)N} = \\
&\frac{1}{Z_{N,\delta_N}^{T_N}e^{-\phi N}}\somme{k=(1-\varepsilon)N}{N}e^{-k\phi}E\Big( e^{-H_{k,\delta_N}^{T_N}} \mathbf{1}\{ k \in \tau^{T_N} \}\Big) P \Big( \tau_1^{T_N} \geq N-k  \Big) e^{-(N-k)\phi}. 
\end{aligned}
\end{equation}
With \eqref{3.13} and  \eqref{Proposition 1}, $E\Big( e^{-H_{k,\delta_N}^{T_N}} \mathbf{1}\{ k \in \tau^{T_N} \}\Big) e^{-k \phi} = \proba{k \in \tau} \asymp_N {1}/(T_N^3 \delta_N^2)$. Then, by combining Lemma  \ref{Lemme sur toucher en n pour la m.a.s.} and Equation \eqref{g + phi}, we obtain that $ P( \tau_1^{T_N} \geq N-k) e^{-(N-k)\phi} \asymp_N \frac{1}{\sqrt{N-k+1}} + \frac{1}{T_N} $. Thus:
\begin{equation}
\probaPolymere{L_N \geq (1-\varepsilon)N} \asymp_N
\frac{e^{\phi N}}{Z_{N,\delta_N}^{T_N}}\somme{k=(1-\varepsilon)N}{N}
\frac{1}{T_N^3  \delta_N^2} \left( \frac{1}{T_N} + \frac{1}{\sqrt{N-k+1}} \right).
\label{123.6.33}
\end{equation}
Recalling that $T_N^3 \delta_N = \beta N$, we obtain by straightforward computations:
\begin{equation}
\somme{k=(1-\varepsilon)N}{N}
\frac{1}{T_N^3  \delta_N^2} \left( \frac{1}{T_N} + \frac{1}{\sqrt{N-k+1}} \right) \asymp_N \frac{\varepsilon}{T_N \delta_N}.
\label{123.6.34}
\end{equation}
With \eqref{123.6.33} and \eqref{123.6.34}, $\probaPolymere{L_N \geq (1-\varepsilon)N} \asymp_N {\varepsilon}[ T_N \delta_N Z_{N,\delta_N}^{T_N}e^{-\phi N}]^{-1}$. Therefore, comparing this to \eqref{123.6.32}, there exists $C>0$ such that, for all $\varepsilon > 0$:
\begin{equation}
\probaPolymere{L_N \geq (1-\varepsilon)N} \leq C \varepsilon 
\probaPolymere{L_N \leq \frac{\C}{\delta_N^2}} 
 \leq C \varepsilon.
 \label{123.6.36}
\end{equation}
\end{proof}
To finish the proof and compare \eqref{123.11.7} to the behaviour of the SRW, we have to prove the following lemma:
\begin{lemme}
For $a<\frac{1}{2}$ and for $\varepsilon>0$ fixed:
\begin{equation}
\limite{N}{\infty} P \left(\tau^{T_N} \cap [ (1-\varepsilon)N,N ] \neq \emptyset \right) = 1.
\end{equation}
\label{lemme 123.6.6}
\end{lemme}

\begin{proof} With Lemma  \ref{Lemme sur toucher en n pour la m.a.s.}:
\begin{equation}
\begin{aligned}
&P\left(\tau^{T_N} \cap [ (1-\varepsilon)N,N ] = \emptyset \right) =
\somme{k=0}{(1-\varepsilon)N} P \left( k \in \tau^{T_N} \right) P \left(\tau_1^{T_N} \geq N-k \right) \\
&\asymp_N
\somme{k=0}{(1-\varepsilon)N} \left( \frac{1}{T_N} + \frac{1}{\sqrt{k}} \right) \frac{1}{T_N} e^{-c\frac{N-k}{T_N^2}} \leq N e^{-c \varepsilon \frac{N}{T_N^2}} \ll 1. 
\end{aligned}
\end{equation}
\end{proof}

\appendix
\section{Asymptotic renewal estimates}
\subsection{Asymptotic free energy estimates \label{Annex A.1} }
In this section we prove \eqref{3.3000}, \eqref{3.4000} and \eqref{phi}. We would like to stress first that our convention for the sign of the repulsion parameter is opposite to that in \cite{Caravenna2009depinning}. By \cite[Theorem 1]{Caravenna_2009}, $Q_{T}(\phi(\delta, T))=e^{\delta}$. Moreover:

\begin{equation}
Q_{T}(\lambda)=1+\sqrt{e^{-2 \lambda}-1} \cdot \frac{1-\cos \Big(T \arctan \sqrt{e^{-2 \lambda}-1}\Big)}{\sin \Big(T \arctan \sqrt{e^{-2 \lambda}-1}\Big)},
\end{equation}
which comes, for example, from \cite[Eq. (A.5)]{Caravenna_2009}. If we denote
\begin{equation}
\gamma=\gamma(\delta_N, T_N):=\arctan \sqrt{e^{-2 \phi(\delta_N, T_N)}-1},
\label{definition de gamma}
\end{equation}
we can therefore write
\begin{equation}
\widetilde{Q}_{T_N}(\gamma(\delta, T_N))=e^{\delta_N} \quad \text { where } \quad \widetilde{Q}_{T_N}(\gamma)=1+\tan \gamma \cdot \frac{1-\cos (T_N \gamma)}{\sin (T_N \gamma)}.
\label{Q tilde}
\end{equation}
Note that $\gamma \mapsto \widetilde{Q}_{T_N}(\gamma)$ is an increasing function with $\widetilde{Q}_{T_N}(0)=1$ and $\widetilde{Q}_{T_N}(\gamma) \rightarrow+\infty$ when $\gamma \uparrow {\pi}/{T_N}$, hence $0<\gamma(\delta_N, T_N)<{\pi}/{T_N}$. We therefore have to study the equation $\widetilde{Q}_{T_N}(\gamma)=e^{\delta_N}$ for $0<\gamma<{\pi}/{T_N}$. An asymptotic study shows that
\begin{equation}
(1+o_N(1)) \gamma \cdot \frac{1-\cos (T_N \gamma)}{\sin (T_N \gamma)}=\delta_N (1+o_N(1)).
\end{equation}
Noting that $x=T_N \gamma$, we get:
\begin{equation}
(1+o_N(1)) x \cdot \frac{1-\cos x}{\sin x}=T_N  \delta_N (1+o_N(1)),
\label{x}
\end{equation}
where $0<x<\pi$. Three cases are now emerging. According to whether $a$ is equal, larger or smaller than $b$, the right hand side of \eqref{x} tends to a constant, $+ \infty$ or 0 respectively.
\begin{itemize}
    \item When $a<b$, the right-hand side in \eqref{x} tends to 0. We can therefore expand further:
\begin{equation}
x^2(1+o_N(1)) = 2 T_N \delta_N (1+o_N(1)).
\end{equation}
Hence $x = \sqrt{2 T_N \delta_N}(1+o_N(1))$, and since $\gamma(\delta,T) = {x}/{T}$,
\begin{equation}
    \gamma(\delta_N,T_N) = \sqrt{ \frac{2\delta_N}{T_N}}(1+o_N(1)).
\label{gamma a<b}
\end{equation}
Recalling \eqref{definition de gamma},
\begin{equation}
\sqrt{e^{-2 \phi(\delta_N, T_N)}-1}=\tan \Big( \sqrt{2 \frac{\delta_N}{T_N}}(1+o_N(1)) \Big).
\end{equation}
Because the function $\lambda \mapsto \arctan \sqrt{e^{-2 \lambda}-1}$ is decreasing and continuously differentiable with non-zero first derivative:
\begin{equation}
\phi(\delta_N,T_N) = \frac{\delta_N}{T_N}(1 + o_N(1)).
\end{equation}
\item When $a=b$ and with the same ideas as in the previous case:
\begin{equation}
x = x_{\beta}(1+o_N(1)), \quad \text{ where } x_\beta = \frac{\sin(x_\beta)\beta}{1-\cos(x_\beta)}, 
\label{def x_beta}   
\end{equation}
\begin{equation}
    \gamma(\delta_N,T_N) = \frac{x_\beta}{T_N}(1+o_N(1)).
\label{gamma a=b}
\end{equation}
and
\begin{equation}
\phi(\delta_N,T_N) = - \frac{x_\beta^2}{2 T_N^2}(1 + o_N(1)).
\end{equation}
\item When $a>b$ the idea is again the same but we give more details that are going to be useful later on:
\begin{equation}
(1+o_N(1)) \frac{2 \pi}{\pi-x}=T_N \delta_N,
\label{A.16}
\end{equation}
\begin{equation}
\gamma(\delta_N, T_N)=\frac{\pi}{T_N}- \frac{2 \pi }{\delta_N T_N^{2}}(1+o_N(1)),
\label{gamma a>b}
\end{equation}
\begin{equation}
\phi(\delta_N, T_N)= - \frac{\pi^2}{2 T_N^2} \left( 1 - \frac{4}{T_N \delta_N}
(1 + o_N(1)) \right).
\label{phi a>b}
\end{equation}

\end{itemize}

\subsection{Estimates of the return time to interfaces \label{Appendix A.2} }

In this section we prove Lemma~\ref{lem:5.3}. We are now looking for asymptotic estimates of the variables $\left(\tau_{1}, \varepsilon_{1}\right)$ under $\probaRenouvellementSansParenthese$ defined in \eqref{2.4}, when $N \rightarrow \infty$, with $T_N = N^b$, $\delta_N = {\beta}{N^{-a}}$ and $a>b$. Let us first focus on $Q_{T_N}^{1}(\phi(\delta_N, T_N))$, where $Q_{T_N}^{1}(\lambda):=E \Big(e^{-\lambda \tau_{1}^{T_N}} \mathbf{1}_{\{\varepsilon_{1}^{T_N}=1\}} \Big)=\somme{n \in \mathbb{N}}{} e^{-\lambda n} q_{T_N}^{1}(n)$. With the same ideas as in the calculation above and thanks to \cite[Eq. (A.5)]{Caravenna_2009}, we can write
\begin{equation}
Q_{T_N}^{1}(\phi(\delta_N, T_N))=\widetilde{Q}_{T_N}^{1}(\gamma(\delta_N, T_N)), \quad \text { where } \quad \widetilde{Q}_{T_N}^{1}(\gamma):=\frac{\tan \gamma}{2 \sin (T_N \gamma)},
\end{equation}
and, according to \eqref{gamma a>b}, we obtain, when $N \rightarrow \infty$:
\begin{equation}
Q_{T_N}^{1}(\phi(\delta_N, T_N))=\frac{\pi}{T_N} \frac{1}{2 \times \frac{2 \pi}{e^{\delta_N}-1} \frac{1}{T_N}}(1+o_N(1)) \sim_N \frac{\delta_N}{4}.
\end{equation}
In particular, thanks to \eqref{définition de la mesure du renouvellement}, we may write when $N \rightarrow \infty$:
\begin{equation}
\EsperanceRenouvellement{\varepsilon_{1}^{2}}=2 \proba{\varepsilon_{1}=+1}=2 e^{-\delta_N} Q_{T_N}^{1}(\phi(\delta_N, T_N)) \sim_N \frac{\delta_N}{2}.
\end{equation}
This proves \eqref{probabilité de changer d'interface}. Now, let us determine the asymptotic behaviour of  $\EsperanceRenouvellement{\tau_{1}}$ when $N \rightarrow \infty$. Thanks to \eqref{définition de la mesure du renouvellement}, we can write:
\begin{equation}
\EsperanceRenouvellement{\tau_{1}}=e^{-\delta_N} \sum_{n \in \mathbb{N}} n q_{T_N}(n) e^{-\phi(\delta_N, T_N) n}=-e^{-\delta_N} \cdot Q_{T_N}^{\prime}(\phi(\delta_N, T_N)), 
\label{A.7}
\end{equation}
\begin{equation}
\EsperanceRenouvellement{\tau_{1}^{2}}=e^{-\delta_N} \sum_{n \in \mathbb{N}} n^{2} q_{T_N}(n) e^{-\phi(\delta_N, T_N) n}=
e^{-\delta_N} \cdot Q_{T_N}^{\prime \prime}(\phi(\delta_N, T_N)).
\end{equation}
Thus, we have to determine $Q_{T_N}^{\prime}(\lambda)$ for $\lambda=\phi(\delta_N, T_N)$. Using the function $\gamma(\lambda):=$ $\arctan \sqrt{e^{-2 \lambda}-1}$ defined in \eqref{definition de gamma} and recalling \eqref{Q tilde}, it follows from $Q_{T_N}=\widetilde{Q}_{T_N} \circ \gamma$ that
\begin{equation}
Q_{T_N}^{\prime}(\lambda)=\widetilde{Q}_{T_N}^{\prime}(\gamma(\lambda)) \cdot \gamma^{\prime}(\lambda), 
\end{equation}
\begin{equation}
Q_{T_N}^{\prime \prime}(\lambda)=\gamma^{\prime \prime}(\lambda) \cdot \widetilde{Q}_{T_N}^{\prime}(\gamma(\lambda))+\left(\gamma^{\prime}(\lambda)\right)^{2} \cdot \widetilde{Q}_{T_N}^{\prime \prime}(\gamma(\lambda)).
\end{equation}
By a direct computation,
\begin{equation}
\widetilde{Q}_{T_N}^{\prime}(\gamma) =\frac{1 - \cos (T_N \gamma)}{\sin (T_N \gamma)} \cdot\left(\frac{1}{\cos ^{2} \gamma}+\frac{T_N\tan \gamma}{\sin (T_N\gamma)}\right), 
\end{equation}
\begin{equation}
\begin{aligned}
\widetilde{Q}_{T_N}^{\prime \prime}(\gamma) =&\frac{1-\cos (T_N\gamma)}{\sin (T_N\gamma)} \left(\frac{2 T_N}{\sin (T_N\gamma) \cos ^{2} x}+\frac{2 \sin \gamma}{\cos ^{3} x}+\frac{T_N^{2} \tan \gamma}{\sin ^{2}(T_N\gamma)}(1-\cos (T_N\gamma))\right),
\end{aligned}
\end{equation}
and
\begin{equation}
\gamma^{\prime}(\lambda)=-\frac{1}{\sqrt{e^{-2 \lambda}-1}}, \quad \gamma^{\prime \prime}(\lambda)=-\frac{e^{-2 \lambda}}{\left(e^{-2 \lambda}-1\right)^{3 / 2}}.
\end{equation}
Thanks to \eqref{A.7} and \eqref{definition de gamma}:
\begin{equation}
\EsperanceRenouvellement{\tau_{1}}=-e^{-\delta_N} \cdot \widetilde{Q}_{T_N}^{\prime}(\gamma(\delta_N, T_N)) \cdot \gamma^{\prime}(\phi(\delta_N, T_N)).
\end{equation}
\noindent The asymptotics in \eqref{gamma a>b} and \eqref{phi a>b} give
\begin{equation}
\begin{aligned}
&\widetilde{Q}_{T_N}^{\prime}(\gamma(\delta_N, T_N))= \frac{\delta_N^2 T_N^2}{2 \pi}(1 + o_N(1)), \\
&\gamma^{\prime}(\phi(\delta_N, T_N))= - \frac{T_N}{\pi}(1 + o_N(1)),
\end{aligned}
\end{equation}
\noindent and
\begin{equation}
\widetilde{Q}_{T_N}^{\prime \prime}(\gamma(\delta_N, T_N))= \frac{T_N^4 \delta_N^3}{2\pi^2}(1+o_N(1)), \quad \gamma^{\prime \prime}(\phi(\delta_N, T_N))=  
- \left( 
\frac{T_N}{\pi}
\right)^{3}(1 + o_N(1)).
\end{equation}
By combining the preceding relations:
\begin{equation}
\EsperanceRenouvellement{\tau_{1}}=   \frac{T_N^3 \delta_N^2}{2 \pi^2}(1 + o_N(1)) \quad \text{ and } \quad 
  \EsperanceRenouvellement{\tau_{1}^{2}}=  \frac{T_N^6 \delta_N^3}{2\pi^4}
  (1 + o_N(1)), 
\end{equation}
which proves \eqref{espérance de tau 1} and \eqref{moment d'ordre 2 de tau 1}.
\section{ Simple random walk estimate on the probability to visit an interface \label{Annexe B.1}}
In this section we prove Lemma \ref{lemme 8.1}. We  prove the result with $2n$ instead of $n$ to simplify notation.  If $2n\leq T$ then \eqref{Bourgogne 22} follows easily from $P(2n \in \tau^T) = \binom{2n}{n} \frac{1}{2^{2n}} \sim \frac{c}{\sqrt{n}}$. Else, for $k \in  \frac{1}{2}\mathbb{Z}$:
\begin{equation}
\begin{aligned}
P(S_{2n} = 2kT) &= \binom{2n}{n+kT}\frac{1}{2^{2n}}  \asymp_n \frac{n^{2n}}{\sqrt{n} (n+kT)^{n+kT}(n-kT)^{n-kT} } \\
& \asymp_n \frac{1}{\sqrt{n}}\exp\left( -(n-kT) \ln(1-kT/n) - (n+kT)\ln(1+kT/n)  \right).
\end{aligned}
\label{B.1}
\end{equation}
Moreover, $(1-x)\ln(1-x) + (1+x)\ln(1+x) \sim_0 x^2 $. Therefore,  there exists $ \varepsilon > 0$ such that, when $\left| \frac{kT}{n} \right| \leq \varepsilon$,
\begin{equation}
\exp\left( (n-kT) \ln(1-kT/n) + (n+kT)\ln(1+kT/n)  \right) \asymp_n \exp \left(-\frac{(kT)^2}{n}\right).
\end{equation}
Now, we need to bound $P(S_n \geq \varepsilon n)$ from above.
\begin{lemme}
When $\varepsilon > 0$ , one has:
\begin{equation}
    P(S_n \geq \varepsilon n) \leq \exp\left(-\frac{n}{2}(\varepsilon^2 + O(\varepsilon^3))\right).
\end{equation}
\end{lemme}
\begin{proof} The exponential Markov inequality with $\lambda = \frac{1}{2}\log\left( 1 + \frac{2\varepsilon}{1-\varepsilon} \right)$ combined with a second-order expansion of $\log$ yields:
\begin{equation}
\begin{aligned}
P(S_n \geq \varepsilon n) \leq \frac{E\left( e^{\lambda S_n} \right)}{e^{\lambda \varepsilon n}} &\leq \exp \left( n\left( (1-\varepsilon)\lambda + \log \left( \frac{1 + e^{-2\lambda}}{2} \right) \right) \right) \\
&\leq \exp\left(-\frac{n}{2}\left(\varepsilon^2 + O \left(\varepsilon^3 \right) \right) \right).
\end{aligned}
\end{equation}
\end{proof}
Hence, when $\varepsilon$ is small enough: 
\begin{equation}
P(|S_n| \geq \varepsilon n) \leq 2\exp\left( -\frac{n \varepsilon^2}{3} \right) \ll \frac{1}{\min\{T, \sqrt{n}\}}.
\label{B.5}
\end{equation}
Now, let us compute $\somme{k=0}{ \varepsilon n/T} \exp \left( - \frac{(kT)^2}{n}\right) $. By standard techniques:
\begin{equation}
\begin{aligned}
\somme{k=0}{\varepsilon n/T} \exp \left( - \frac{(kT)^2}{n}\right) 
\asymp_n 1 + \int_0^{\frac{\varepsilon n}{T}} \exp \left( - \frac{(Tt)^2}{n}\right) dt
&\asymp_n  1 + \frac{\sqrt{n}}{T} \int_0^{\sqrt{n} \varepsilon} e^{-t^2}dt \\ &\asymp_n 1 + \frac{\sqrt{n} }{T} \asymp_n \frac{\max\{ T,\sqrt{n} \}}{T}.
\end{aligned}
\label{B.6}
\end{equation}
Therefore, when $\varepsilon$ is small enough, thanks to \eqref{B.1}, \eqref{B.5} and \eqref{B.6}:
\begin{equation}
\begin{aligned}
P(S_{2n} \in T\mathbb{Z}) \asymp_n  \somme{k= - \frac{ \varepsilon n}{T}}{\frac{ \varepsilon n}{T}} P(S_{2n} = kT) + P(|S_{2n}| \geq \varepsilon n ) &\asymp_n \frac{\max\{T,\sqrt{n}\}}{T \sqrt{n}} \\&\asymp_n \frac{1}{\min\{\sqrt{n},T\}}.
\end{aligned}
\end{equation}
Hence, Lemma \ref{lemme 8.1} is proven.
\section{Proof of technical results \label{Annexe A}}

\subsection{Precise estimates for the simple random walk: Lemma~\ref{lemme 1.1}}

We recall the extended Stirling formula \cite[ 6.1.37]{Abrahamovitz}:
\begin{equation}
n! = \left(\frac{n}{e}\right)^n\sqrt{2\pi n}\left( 1 + \frac{1}{12 n} + O_n\left( \frac{1}{n^2} \right) \right).
\end{equation}
Using that $P \left(2n \in \tau^\infty \right) = \binom{2n}{n}\frac{1}{4^n}$, we obtain:
\begin{equation}
P \left(2n \in \tau^\infty \right) \sim
\frac{1}{\sqrt{n\pi}} \frac{1 + \frac{1}{24n} + O_n \left( \frac{1}{n^2} \right)}{\left( 1 + \frac{1}{12n} + O_n \left( \frac{1}{n^2} \right) \right)^2}   \sim \frac{1}{\sqrt{n\pi}} \left( 1  - \frac{1}{8n} + O_n\left( \frac{1}{n^2} \right) \right).
\end{equation}
By evaluating in $n$ instead of $2n$, Equation \eqref{1.1000} is proven. Equation \eqref{1.2000} follows, because $P\left(\tau_1^\infty = n \right) = \frac{1}{n-1}P \left(n \in \tau^\infty \right)$ by \cite[Eq. (3.7)]{kesten2008introduction}.
\subsection{Computation of an integral: Lemma~\ref{lemme 1.2}}
Let us change first $t$ by $\sin^2(t)$. Then, the primitive is $\frac{- \cos(2t)}{\sin(2t)}$, and because $\cos(2\arcsin(u)) = 1 - 2u^2$, and $\sin(2 \arcsin{u}) = 2u \sqrt{1-u^2}$, it comes that:
\begin{equation}
\begin{aligned}
\int_\frac{1}{2}^{1-\varepsilon} \frac{dt}{t^{3/2} (1-t)^{3/2}}
= 
8 \int_{\arcsin{\sqrt{1/2}}}^{\arcsin{\sqrt{1-\varepsilon}}} \frac{dt}{\sin^2(2t)} = \frac{2-4\varepsilon}{ \sqrt{1-\varepsilon}\sqrt{\varepsilon}}
 &= \frac{2-4\varepsilon}{ \left(1 - \frac{\varepsilon}{2} + O_{\varepsilon}(\varepsilon^2) \right)\sqrt{\varepsilon}} \\ &= \frac{2 - 3\varepsilon + O_{\varepsilon}(\varepsilon^2)}{\sqrt{\varepsilon}}.
\end{aligned}
\end{equation}
\subsection{Simple random walk tail estimates for the time to hit the origin}
Let us prove an equivalent result, that is $P \left( \tau_1^\infty > 2l \right) = \frac{1}{\sqrt{\pi l} } + \frac{3}{8 l^{3/2}} + o_l \left( \frac{1}{l^{3/2}}\right)$. Thanks to \eqref{1.2000}:
\begin{equation}
\begin{aligned}
P \left( \tau_1^\infty > 2l \right) 
&=
\frac{\sqrt{2}}{\sqrt{\pi}}\somme{p=l}{\infty}
\frac{1 + \frac{3}{8p} + o_p\left(\frac{1}{p} \right) }{(2p)^{3/2}} \\
&= \frac{1}{\sqrt{\pi l} } + \frac{1}{2\sqrt{\pi}}
\somme{p=l}{\infty} \left( \frac{1}{p^{3/2}} - \int_p^{p+1} \frac{dt}{t^{3/2}} + \frac{3}{8p^{5/2}} \right) 
+ o_l \left( \frac{1}{l^{3/2}}\right)\\
&= 
\frac{1}{\sqrt{\pi l}} + o_l \left( \frac{1}{l^{3/2}} \right) + \frac{1}{2\sqrt{\pi}} \somme{p=l}{\infty} \frac{9}{8p^{5/2}}.
\end{aligned}
\end{equation}
\noindent Hence the result we were looking for.
\subsection{An estimate of the occurrences of interface switching \label{Annexe C.45} }
{\red{In this section we prove Lemma \ref{Lemma 7.9}. We denote $A := \inf \{i \geq 1 : \varepsilon_i^2 = 1\}$. By independence, 
\begin{equation}
  \proba{\D{k}{m}} \geq \proba{\exists i : \varepsilon_i^2 = 1, \tau_i \leq k/m}^m \proba{\tau_A>N} .
\end{equation}
Remind that $N =  T_N^3 \delta_N/\beta $.}} A computation gives:
\begin{equation}
\begin{aligned}
&\proba{\exists i : \varepsilon_i^2 = 1 ; \tau_i \leq k/m} \\
&\geq \somme{l=1}{1/\delta_N^2} \proba{l \in \tau} \proba{\tau_1 \in \left[ \frac{T_N^3 \delta_N}{8 m \beta}, \frac{3T_N^3 \delta_N}{8 m \beta} \right], \varepsilon_i^2 = 1} 
\end{aligned}
\end{equation}
Using \eqref{123.4}, we can see that
\begin{equation}
\proba{\tau_1 = n, 
\varepsilon_1^2 = 1} \geq \frac{ c_{11}}{2T_N^3}e^{-n (g(T_N) + 
\phi(\delta_N,T_N)}
\label{Patate C.7}
\end{equation}
when $n \geq M T_N^2$, for a certain $M>0$. 
Hence, when $\frac{T_N^3 \delta_N}{4 m \beta} \geq T_N^2 M $, which 
happens eventually because $a>b$, one can see that
\begin{equation}
\proba{\tau_1 
\in \left[ \frac{T_N^3 \delta_N}{8 m \beta}, \frac{3 T_N^3 \delta_N}
{8 m \beta} \right], \varepsilon_1^2 = 1} \geq \frac{c  \delta_N }{m},
\label{Patate C.8}
\end{equation} 
with $c \geq \frac{c_{11} e^{-1/\beta}}{8\beta}$ a constant independent of $m$. Using Equation \eqref{Proposition 1}, we got that $\proba{\exists i : \varepsilon_i^2 = 1 ; \tau_i \leq k/m} \geq c/m$. 
Now, with computations similar to \eqref{Patate C.7} and \eqref{Patate C.8}, one can see that $\proba{\tau_1 \geq N, \varepsilon_1^2 = 1} \geq c' \delta_N$. We now have, by \eqref{probabilité de changer d'interface}:
\begin{equation}
\proba{\tau_A \geq N} \geq \somme{k=0}{\infty} \proba{\varepsilon_1^2 =0}^k \proba{\tau_T \geq N} \geq c' \delta_N \somme{k=0}{} 
(1- \delta_N/2)^k = c'.
\end{equation}
\noindent Lemma \ref{Lemma 7.9} is therefore proven.
\subsection{ Technical upper bounds \label{Appendix E}}
In this section we prove Lemma~\ref{lemme 5.1}. Equation \eqref{5.6000} easily follows by using \eqref{probabilité que tau 1 = n} and alleviating the first constraint in the sum on the left-hand side. Equation  \eqref{5.7} is somewhat more complicated. We have to cut the sum at $l=T_N^2$. 

\begin{itemize}
    \item For $l \leq T_N^2$, we use that $g(T_N) \sim  {c}/{T_N^2}$ so $e^{g(T_N) T_N^2} \sim c$. We set, with the constant $C$ changing at each line and being independent of all other variable,  $R(n,T_N) :=  {Ce^{-g(T_N)(n-l_2-...-l_j)}}/{T_N^3} $. Using  \eqref{probabilité pour la marche aléatoire simple que tau 1 vaille n}:
\begin{equation}
\begin{aligned}
&\somme{l=1}{  T_N^2  }Q_{k-j}^{T_N}(l) P \left(\tau_1^{T_N} 
= n-l-l_1-...-l_j \right) \leq R(n,T_N) \somme{l=1}{  T_N^2  }Q_{k-j}^{T_N} (l) \\&\leq R(n,T_N)
P\left(\tau_{k-j}^{T_N} \leq T_N^2 \right) \leq R(n,T_N).
\end{aligned}
\label{B.1000}
\end{equation}
\item For $l\geq T_N^2$, we can use \eqref{8.6}, which leads to the following upper bound:
\begin{equation}
\begin{aligned}
&\somme{l=T_N^2}{ n-l_2-...-l_j-T_N^2  }Q_{k-j}^{T_N}(l) P\left(\tau_1^{T_N} = n-l-l_2-...-l_j\right) \\&\leq \frac{R(n,T_N)}{T_N^3} \left(1 + \frac{C}{T_N} \right)^k \somme{l=1}{n} 1 \leq R(n,T_N) \left(1 + \frac{C}{T_N} \right)^k ,
\end{aligned}
\label{B.2000}
\end{equation}
because $n \leq T_N^3$. By summing \eqref{B.1000} and \eqref{B.2000}, we have proven \eqref{5.7}.
\end{itemize}
\section{ Estimate of the partition function \label{Annex F}}
Recall the definition of $L_N$ in Definition~\ref{definition 1}. By the Markov property, together with \eqref{3.13} and   \eqref{majoration de la proba que tau 1 soit > m}:
\begin{equation}
\begin{aligned}
&Z_{N, \delta_N}^{T_{N}}  =E\left[e^{-H_{N, \delta}^{T_{N}}(S)}\right]=\sum_{r=0}^{N} E\left[e^{-H_{N, \delta}^{T_{N}}(S)} \mathbf{1}_{\left\{ \tau_{L_N}^{T_N} =r\right\}}\right] \\
& =\sum_{r=0}^{N} E\left[e^{-H_{r, \delta}^{T_{N}}(S)} \mathbf{1}_{\left\{r \in \tau^{T_{N}}\right\}}\right] P\left(\tau_{1}^{T_{N}} \geq N-r\right) \\
& =\sum_{r=0}^{N} e^{\phi\left(\delta_N, T_{N}\right) r} \proba{r \in \tau^{T_N}} P\left(\tau_{1}^{T_{N}} \geq N-r\right) \\
&\asymp_N 
e^{\phi\left(\delta_N, T_{N}\right) N} 
\sum_{r=0}^{N} 
e^{(\phi\left(\delta_N, T_{N}\right) + g(T_N))(N- r)} \proba{r \in \tau} \left( 
\frac{1}{\sqrt{N-r}} + \frac{1}{T_N} \right).
\end{aligned}
\label{123.4.14}
\end{equation}
We hereafter distinguish between several cases, and abbreviate $\phi := \phi(\delta_N,T_N)$ as well as $g:=g(T_N)$.
\begin{enumerate}
    \item If $a \leq b$, \eqref{Proposition 123} gives $\proba{r \in \tau} \asymp_N \frac{1}{T_N} + \frac{1}{\sqrt{r}} $. Moreover, \eqref{g + phi, cas b geq a} gives $g+\phi \asymp_N \frac{1}{T_N^2}$. Plugging this in \eqref{123.4.14} gives us that  $Z_{N, \delta_N}^{T_{N}} \asymp_N e^{N \phi(\delta_N,T_N)} $.
    \item It $a > b$, \eqref{g + phi} gives $g+\phi \asymp_N \frac{1}{T_N^3 \delta_N}$. Remember  \eqref{Proposition 1} and the three types of behaviour of $\proba{r \in \tau}$. We then distinguish again between several cases:

    \begin{enumerate}
        \item If $T_N^3 \delta_N \ll N$, the term which gives the main contribution in the sum is for $ N - T_N^3 \delta_N \leq  r \leq N$. There, $\proba{r \in \tau} \asymp_N \frac{1}{T_N^3 \delta_N^3}$, so the whole sum is of order $\frac{1}{T_N \delta_N}$.
        
        \item If  $T_N^3 \delta_N \geq N \ge T_N^2$, we can neglect the term $e^{(\phi + g)(N-r)}$. Splitting the sum when $1 \leq r \leq {1}/{\delta_N^2}$, ${1}/{\delta_N^2} \leq r \leq T_N^2$ and $T_N^2 \leq r \leq N$ gives us that  the first two terms are the one contributing the most, with a contribution of order $\frac{1}{T_N \delta_N}$.
        
        \item If ${1}/{\delta_N^2} \leq N \leq T_N^2$, we now have to split the sum in two parts: $1 \leq r \leq {1}/{\delta_N^2}$ and ${1}/{\delta_N^2} \leq r \leq N $. The first sum contributes the most, with a contribution of order $\frac{1}{\delta_N \sqrt{N}}$.
        
        \item If $N \leq {1}/{\delta_N^2}$, then $\proba{r \in \tau} \asymp_N {1}/{\sqrt{r}}$. Hence, the sum is of order $\somme{r=1}{N-1} \frac{1}{\sqrt{r} \sqrt{N-r}} $, that is of constant order.
    \end{enumerate}
\end{enumerate}
\section{Asymptotic estimates on the time to hit another interface, diffusive case \label{Annex E.22} }
To prove Lemma \ref{lemme 123.5.654}, we follow the ideas from the proof in \cite[Appendix A.2]{Caravenna2009depinning}. Let us remind the reader that, from  \eqref{gamma a<b} and  \eqref{gamma a=b}:
\begin{equation}
\gamma = 
\left\{
\begin{array}{l}
  \sqrt{ \frac{2\delta_N}{T_N}}(1+O_N(\min\{ T_N^3 \delta_N^3, \delta_N \})),\quad \text{ if $b>a$} ,\\
  \frac{x_\beta}{T_N}(1 + o_N(1)), \quad
  \text{ if $b=a$}.\\
\end{array}
\right.   
\end{equation}
Moreover, to prove \eqref{moment d'ordre 2 de tau 1, b geq a}, one needs the following approximation, when $b>a$: 
\begin{equation}
\phi(\delta_N,T_N) = -\frac{\delta_N}{T_N}\left( 1 + o \left( \max\{ \delta_N, T_N^3 \delta_N^3 \} \right) \right).
\end{equation}
Hence, the equation above (A.5) in \cite{Caravenna2009depinning} gives us that 
\begin{equation}
\Tilde{Q}_{T_N}(\gamma) =
\left\{
\begin{array}{l}
    \frac{1}{2 T_N}(1+o_N(1)),\quad  \text{ if $b>a$,}  \\
    \frac{x_\beta}{\sin(x_\beta) T_N }(1+o_N(1)), \quad \text{ if $b=a$. }
\end{array}
\right.
\end{equation}
Using (A.6) in \cite{Caravenna2009depinning}, i.e. $\proba{\varepsilon^2=1} = 2 e^{-\delta_N}\Tilde{Q}_{T_N}^1(\phi(\delta_N,T_N))$, we have easily proven \eqref{Probabilité de changer d'interface,b geq a}. Now, using (A.11) and (A.12) in \cite{Caravenna2009depinning}, and remembering that the $x$ in \cite{Caravenna2009depinning} is $\gamma/T$:
\begin{equation}
\Tilde{Q}_{T_N}'(\gamma) =
\left\{
\begin{array}{l}
    \sqrt{2T_N \delta_N}\left( 1 + \frac{T_N \delta_N}{3} \right)(1+o_N(1)), \quad \text{ if $b>a$}  \\
     \frac{\beta}{x_\beta} \left( 1 + \frac{x_\beta}{\sin (x_\beta)}  + o_{N}(1) \right),\quad  \text{ if $b=a$, }
\end{array}
\right.
\end{equation}
\begin{equation}
\Tilde{Q}_{T_N}''(\gamma) =
\left\{
\begin{array}{l}
     T_N(1 + \delta_N T_N)(1+o_N(1)),\quad \text{ if $b>a$}  \\
     \frac{T_N \beta}{x_\beta} \left( 
     \frac{2}{\sin(x_\beta)} + \frac{\beta}{\sin(x_\beta)} + o_{T_N}(1) 
    \right),\quad \text{ if $b=a$, }
\end{array}
\right.
\end{equation}
\begin{equation}
\gamma'(\phi(\delta_N,T_N)) =
\left\{
\begin{array}{l}
     \left( \frac{T_N}{2 \delta_N} \right)^{1/2}\left( 1 + o_N(\max \{\delta_N,T_N^3 \delta_N^3 \}) \right), \quad \text{ if $b>a$}  \\
     - \frac{T_N}{x_\beta}(1+o_{T_N}(1)),\quad \text{ if $b=a$, }
\end{array}
\right.
\end{equation}
\begin{equation}
\gamma''(\phi(\delta_N,T_N)) =
\left\{
\begin{array}{l}
     \left( \frac{T_N}{2 \delta_N} \right)^{3/2}\left( 1 + o_N(\max \{\delta_N,T_N^3 \delta_N^3 \}) \right),\quad \text{ if $b>a$}  \\
     - \frac{T_N^3}{x_\beta^3}(1+o_{T_N}(1)),\quad \text{ if $b=a$. }
\end{array}
\right.
\end{equation}
Combining these equations with (A.7) and (A.8) in \cite{Caravenna2009depinning}, we get \eqref{moment d'ordre 1 de tau 1, b geq a} and \eqref{moment d'ordre 2 de tau 1, b geq a}.

\bigskip

\noindent \textbf{Acknowledgment} The author would like to thank J. Poisat and N. Pétrélis for the long talks and inspiring ideas about this paper.

\bigskip
\noindent \textbf{Conflict of Interest} The author has no relevant financial or non-financial interests to disclose.

\bigskip

\noindent \textbf{Data availability} No datasets were generated or analysed during the current study.

\bibliography{main.bib}

\end{document}